\documentclass[notitlepage,11pt,reqno]{amsart}
\usepackage[foot]{amsaddr}
\usepackage{amssymb,nicefrac,bm,upgreek,mathtools,verbatim}
\usepackage[final]{hyperref}
\usepackage[mathscr]{eucal}
\usepackage{dsfont}
\usepackage[normalem]{ulem}
\usepackage{amsopn}

\usepackage[margin=1in]{geometry}
\allowdisplaybreaks
\newcommand{\stkout}[1]{\ifmmode\text{\sout{\ensuremath{#1}}}\else\sout{#1}\fi}

\newtheorem{lemma}{Lemma}[section]
\newtheorem{theorem}{Theorem}[section]

\theoremstyle{definition}

\newtheorem{assumption}{Assumption}[section]

\newtheorem{example}{Example}[section]

\theoremstyle{remark}

\numberwithin{theorem}{section}
\numberwithin{equation}{section}

\hypersetup{
  colorlinks=true,
  citecolor=mblue,
  linkcolor=mblue,
  urlcolor = blue,
  anchorcolor = blue,
  frenchlinks=false,
  pdfborder={0 0 0},
  naturalnames=false,
  hypertexnames=false,
  breaklinks}
\usepackage[capitalize,nameinlink]{cleveref}
\usepackage[abbrev,msc-links,nobysame,citation-order]{amsrefs} 

\crefname{section}{Section}{Sections}
\crefname{subsection}{Section}{Sections}
\crefname{condition}{Condition}{Conditions}
\crefname{hypothesis}{Hypothesis}{Conditions}
\crefname{assumption}{Assumption}{Assumptions}
\crefname{lemma}{Lemma}{Lemmas}
\crefname{fact}{Fact}{Facts}

\Crefname{figure}{Figure}{Figures}

\crefformat{equation}{\textup{#2(#1)#3}}
\crefrangeformat{equation}{\textup{#3(#1)#4--#5(#2)#6}}
\crefmultiformat{equation}{\textup{#2(#1)#3}}{ and \textup{#2(#1)#3}}
{, \textup{#2(#1)#3}}{, and \textup{#2(#1)#3}}
\crefrangemultiformat{equation}{\textup{#3(#1)#4--#5(#2)#6}}%
{ and \textup{#3(#1)#4--#5(#2)#6}}{, \textup{#3(#1)#4--#5(#2)#6}}%
{, and \textup{#3(#1)#4--#5(#2)#6}}

\Crefformat{equation}{#2Equation~\textup{(#1)}#3}
\Crefrangeformat{equation}{Equations~\textup{#3(#1)#4--#5(#2)#6}}
\Crefmultiformat{equation}{Equations~\textup{#2(#1)#3}}{ and \textup{#2(#1)#3}}
{, \textup{#2(#1)#3}}{, and \textup{#2(#1)#3}}
\Crefrangemultiformat{equation}{Equations~\textup{#3(#1)#4--#5(#2)#6}}%
{ and \textup{#3(#1)#4--#5(#2)#6}}{, \textup{#3(#1)#4--#5(#2)#6}}%
{, and \textup{#3(#1)#4--#5(#2)#6}}

\crefdefaultlabelformat{#2\textup{#1}#3}
\newcommand{\vertiii}[1]{{\left\vert\kern-0.25ex\left\vert\kern-0.25ex\left\vert #1 
    \right\vert\kern-0.25ex\right\vert\kern-0.25ex\right\vert}}
\newcommand{\lamstr}{\lambda^{\!*}}

\newcommand{\Uadm}{\mathfrak Z}
\newcommand{\Act}{\mathcal{Z}}
\newcommand{\Usm}{\mathfrak Z_{\sf{sm}}}
\newcommand{\bUsm}{\overline{\mathfrak Z}_{\sf{sm}}}

\newcommand{\cA}{{\mathcal{A}}}  
\newcommand{\cB}{{\mathcal{B}}}  
\newcommand{\sB}{{\mathscr{B}}}  
\newcommand{\cC}{{\mathcal{C}}}   
\newcommand{\sD}{{\mathscr{D}}}  %
\newcommand{\sE}{{\mathscr{E}}} 
\newcommand{\sF}{{\mathfrak{F}}}   
\newcommand{\cG}{{\mathcal{G}}}  
\newcommand{\Hg}{{\mathcal{H}}}  
  
\newcommand{\cK}{{\mathcal{K}}}  
\newcommand{\sK}{{\mathscr{K}}}  
\newcommand{\Lg}{{\mathcal{L}}}    
\newcommand{\Lp}{{L}}            
\newcommand{\cT}{{\mathcal{T}}}
\newcommand{\cP}{{\mathcal{P}}}  
\newcommand{\sP}{{\mathscr{P}}}
\newcommand{\cS}{{\mathcal{S}}}
\newcommand{\Lyap}{{\mathcal{V}}}  
\newcommand{\cX}{{\mathcal{X}}}

\newcommand{\RR}{\mathds{R}}
\newcommand{\NN}{\mathds{N}}

\newcommand{\Rd}{{\mathds{R}^{d}}}
\DeclareMathOperator{\Exp}{\mathbb{E}}
\DeclareMathOperator{\Prob}{\mathbb{P}}
\newcommand{\D}{\mathrm{d}}

\newcommand{\Ind}{\mathds{1}}   
\newcommand{\Vt}{\vartriangle}  
\newcommand{\Sob}{{\mathscr W}}    
\newcommand{\Sobl}{{\mathscr W}_{\text{loc}}} 

\newcommand{\df}{:=}
\newcommand{\transp}{^{\mathsf{T}}}

\DeclareMathOperator*{\trace}{Tr}
\DeclareMathOperator*{\dist}{dist}

\newcommand{\order}{{\mathscr{O}}}
\newcommand{\sorder}{{\mathfrak{o}}}
\newcommand{\grad}{\nabla}
\newcommand{\uuptau}{{\Breve\uptau}}

\newcommand{\abs}[1]{\lvert#1\rvert}
\newcommand{\norm}[1]{\lVert#1\rVert}

\usepackage{color}
\definecolor{dmagenta}{rgb}{.4,.1,.5}
\definecolor{dblue}{rgb}{.0,.0,.5}
\definecolor{mblue}{rgb}{.0,.0,.7}
\definecolor{ddblue}{rgb}{.0,.0,.4}
\definecolor{dred}{rgb}{.7,.0,.0}
\definecolor{dgreen}{rgb}{.0,.5,.0}
\definecolor{Eeom}{rgb}{.0,.0,.5}

\newcommand{\ttl}{\Large Ergodic Risk-Sensitive Control for Regime-Switching Diffusions}
\begin{document}
\title[Risk-sensitive control for Switching Diffusions]
{\ttl}

\author[Anup Biswas]{Anup Biswas$^\dag$}
\address{$^\dag$Department of Mathematics,
Indian Institute of Science Education and Research,
Dr.\ Homi Bhabha Road, Pune 411008, India}
\email{anup@iiserpune.ac.in}

\author[Somnath Pradhan]{Somnath Pradhan$^\ddag$}
\address{$^\ddag$Department of Mathematics and Statistics,
Queen's University, Kingston, ON K7L 3N6, Canada}
\email{sp165@queensu.ca}

\begin{abstract}
In this article, we study the ergodic risk-sensitive control problem for controlled regime-switching diffusions. Under a blanket stability hypothesis, we solve the 
associated nonlinear eigenvalue problem for weakly coupled systems and characterize the optimal stationary Markov controls via a suitable verification theorem.
We also consider the near-monotone case and obtain the existence of principal eigenfunction and optimal stationary Markov controls.
\end{abstract}
\keywords{Principal eigenvalue, switching diffusions, verification results}

\subjclass[2000]{Primary 93E20, 60J60}

\maketitle


\section{Introduction}
In this article, we study the ergodic risk-sensitive stochastic optimal control problem for a large class of multidimensional controlled switching diffusions in $\Rd\times \cS$ where $\cS$ is a finite set
representing states of the random switching.
Since the seminal work \cite{HM-72}, a lot of effort has been devoted to risk-sensitive optimization problems. This is mainly due to its wide range of applications in various applied subjects like mathematical finance \cite{FS00}, large deviation \cite{Kaise-06}, missile guidance \cite{Speyer}, and several other fields. There is a substantial amount of work in the literature on this topic, see for instance \cite{ABBK-19,ABP-21,AB18,AB-19,AB18a,Biswas-11a,Biswas-11,FM95,Menaldi-05} for controlled diffusions, \cite{ABS19} for controlled jump-diffusions, \cite{SP19} for reflecting diffusions. However, to the best of our knowledge, the risk-sensitive optimal control problem in the context of controlled switching diffusions is not yet studied in the literature. In \cite{AGM93}, Ghosh et.\ al.\ have studied the risk-neutral counterpart of this problem\,. 

Though this topic has been around for quite sometime, some of the major issues such as uniqueness of value functions, verification results, variational representations, etc. for the (ergodic) risk-sensitive control problem were resolved fairly recently \cite{ABBK-19, ABS19, AB18, AB-19}, in the context of controlled diffusions. In \cite{ABS19} the authors have realized that these issues are related to the the (strict) monotonicity of principal eigenvalues with respect to the potential (or zeroth order term) of the associated nondegenerate elliptic operators. It is shown in \cite{Berestycki-15} that the monotonicity  of principal eigenvalues with respect to the potential term is strict in the bounded domains and it might not be the case when the domain is unbounded. In view of the above discussions for controlled diffusions, in order to address these issues for controlled switching diffusions one needs to analyze the continuity and monotonicity property of principal eigenvalue of the associated weakly coupled system of elliptic PDEs.

Owing to the demand of modeling, analysis, and computation of complex networked
systems, much attention has been devoted in studying regime switching diffusions. It is heavily used in modelling many practical systems  such as mathematical finance \cite{Merton,BDY09,DKR94}, wireless communications \cite{YKI04}, production planning \cite{SZ94}, predictive modelling \cite{MH12}\,.
In a recent comprehensive work \cite{ABP-21}, the authors have studied the monotonicity property of the generalized eigenvalue for a general linear weakly-coupled cooperative elliptic systems (which is related to the generator of certain regime switching diffusions) with respect to the potential term\,. The notion of monotonicity of principal eigenvalue for scalar differential equation were studied in 
\cite{ABS19} using probabilistic techniques. \cite{ABP-21} improves 
this work not only for weakly-coupled systems but also for a very general class of elliptic PDEs. 

In this article, our main goal is to establish the existence-uniqueness of value function and verification of the optimality of stationary Markov policies for ergodic risk-sensitive stochastic control problem where the underlying dynamics are given by regime switching diffusions in $\Rd \times \cS$ (see \eqref{E1.1}). We consider two situations.
In the first case, the controlled system is assumed to be geometrically ergodic (see Assumption~\ref{A1.1} below) and we provide a full characterization of the optimal control problem (see Theorem~\ref{T1.1} below). In the second case, no stability condition is imposed, but the cost function is assumed to be near-monotone. We solve the nonlinear eigen-equation in this case and find an optimal stationary Markov control as a measurable selector (see Theorem~\ref{T4.1} below).

The remaining part of the article is organized as follows: In the 
next section we introduce the model and our main result under the stability criterion. Section~\ref{S-bounded} deals with the Dirichlet eigenvalue problem in bounded domains whereas the eigenvalue problem 
in $\Rd$ and proof of \cref{T1.1} are discussed \cref{S-risk}. Finally, we consider the near-monotone cost in \cref{N-control}.

\subsection*{Notation:}
We denote by $\uptau(A)$ the \emph{first exit time} of the process
$\{X_{t}\}$ from the set $A\subset\RR^{d}$, defined by
\begin{equation*}
\uptau(A) \,\df\, \inf\,\{t>0\,\colon X_{t}\not\in A\}\,.
\end{equation*}
The open ball of radius $r$ in $\RR^{d}$, centered at the origin,
is denoted by $\sB_{r}$, and we let $\uptau_{r}\df \uptau(\sB_{r})$,
and $\uuptau_{r}\df \uptau(\sB^{c}_{r})$.
Also $\trace S$ denotes the trace of a square matrix $S$.

The term \emph{domain} in $\RR^{d}$
refers to a nonempty, connected open subset of the Euclidean space $\RR^{d}$. 
For a domain $D\subset\RR^{d}$,
the space $\cC^{k}(D)$ ($\cC^{\infty}(D)$), $k\ge 0$,
refers to the class of all real-valued functions on $D$ whose partial
derivatives up to order $k$ (of any order) exist and are continuous.
Also, $\cC^{k}_{b}(D)$ ($\cC_b^{\infty}(D)$) is the class of functions
whose partial derivatives up to order $k$ (of any order) are continuous and bounded
in $D$, and $\cC_{\mathrm{c}}^k(D)$ denotes the subset of $\cC^{k}(D)$,
$0\le k\le \infty$, consisting of functions that have compact support.
In addition, $\cC_{\mathrm{o}}(\Rd)$ denotes the class of continuous
functions on $\Rd$ that vanish at infinity.
The space $\Lp^{p}(D)$, $p\in[1,\infty)$, stands for the Banach space
of (equivalence classes of) measurable functions $f$ satisfying
$\int_{D} \abs{f(x)}^{p}\,\D{x}<\infty$, and $\Lp^{\infty}(D)$ is the
Banach space of functions that are essentially bounded in $D$.
The standard Sobolev space of functions on $D$ whose generalized
derivatives up to order $k$ are in $\Lp^{p}(D)$, equipped with its natural
norm, is denoted by $\Sob^{k,p}(D)$, $k\ge0$, $p\ge1$.
In general, if $\mathcal{X}(Q)$ is a space of real-valued functions on $Q$, $\mathcal{X}_{\mathrm{loc}}$ consists of all functions $f$ such that $f\varphi\in\mathcal{X}$ for every $\varphi\in\cC_{\mathrm{c}}(\mathcal{X})$.
Likewise, we define $\Sobl^{k, p}(D)$. Also, we adopt the notation $\mathcal{X}(Q\times\cS)$ to indicate the space $(\mathcal{X}(Q))^N$, where $N$ is the cardinality of $\cS$, the corresponding norm on $\mathcal{X}(Q\times\cS)$ is defined by 
$\norm{f}_{\mathcal{X}(Q\times\cS)} \,\df\, \sum_{k\in\cS} \norm{f_{k}}_{\mathcal{X}(Q)}$. Let $f\in \cC(\Rd\times\cS)$, then by $f\gg 0$ we mean $f_{k} > 0$ for all $k\in\cS$.

\subsection{Description of the problem}
The controlled switching diffusion process $(X_{t}, S_{t})$ in $\Rd\times \cS$ is governed by the following stochastic differential equations.
\begin{align}\label{E1.1}
\D X_t \, &= \, b(X_t,S_t,Z_t) \D t + \upsigma(X_t,S_t) \D W_t, \nonumber\\
\D S_t \, &= \, \int_{\RR} h(X_t,S_{t^-},Z_t,z)\cP(\D t, \D z),
\end{align}
for $t\geq 0$. Here
\begin{itemize}
\item[(i)] $S_{0}$ is a prescribed $\cS\, \df \, \{1,2,\dots,N\}$ valued random variable;
\item[(ii)] $X_{0}$ is a prescribed $\RR^d$ valued random variable;
\item[(iii)] $W$ is a $d$-dimensional standard Wiener process;
\item[(iv)] $\cP(\D t, \D z)$ is a Poisson random measure on $\RR_{+}\times\RR$ with intensity $\D t\times m(\D z)$, where $m$ is a Lebesgue measure on $\RR$;
\item[(v)] $\cP(\cdot,\cdot), W(\cdot), X_{0}, S_{0}$ are independent;
\item[(vi)] The function $h :\RR^d \times \cS \times \Act \times \RR \to \RR$ is defined by
\begin{equation*}
h(x,i,\xi,z)\,\df\,\begin{cases}
j - i & \text{if}\,\, z\in \Vt_{i,j}(x,\xi),\\
0 & \text{otherwise},
\end{cases}   
\end{equation*} where for $i,j\in\cS$ and fixed $x,\xi$, $\Vt_{i,j}(x,\xi)$ are left closed right open disjoint intervals of $\RR$ having length $m_{i,j}(x,\xi)$.
\item[(vii)] $M \df \, (m_{i,j})_{i,j\in\cS}$, where $m_{i,j}\geq  0$ if $i\ne j$, and $\sum_{j=1}^{N}m_{i,j} = 0$, for all $i\in\cS$, is the transition matrix of the controlled Markov chain $S_{t}$.
\item[(viii)] The control process $\{Z_t\}$ takes values in a compact metric space $\Act$, is progressively measurable with respect to $\sF_t \,\df\, \text{the\ completion\ of\ }
\sigma\{X_s, S_s, s\le t\} \text{\ relative\ to\ } (\sF, \Prob)\,$, and is \emph{non-anticipative}: for each $t \geq 0$,
 $\sigma\{Z_{s};s\le t\}$ is independent of
$\sigma\{W_{s} - W_{t}, \cP(A,B): A\in\cB([s, \infty)), B\in\cB(\RR), s\geq t\}$.
\end{itemize}
The process $Z$ is called an \emph{admissible} control,
and the set of all admissible control is denoted by $\Uadm$. 

We impose the following assumptions to guarantee existence of solution to \cref{E1.1}.
\begin{itemize}
\item[\hypertarget{A1}{{(A1)}}]
\emph{Local Lipschitz continuity:\/}
The function $\upsigma\,=\,\bigl[\upsigma^{ij}\bigr]\colon\RR^{d}\times\cS\to\RR^{d\times d}$,
$b\colon\Rd\times\cS\times\Act\to\Rd$, and $m_{i,j}\colon\Rd\times\Act\to\RR,\,\, i,j\in\cS$, are continuous and locally Lipschitz in $x$ with a Lipschitz constant $C_{R}>0$ depending on $R>0$.
In other words, with $\norm{\upsigma}\df\sqrt{\trace(\upsigma\upsigma\transp)}$,
we have
\begin{equation*}
\abs{b(x,k,\xi)-b(y,k,\xi)}^2 + \norm{\upsigma(x,k) - \upsigma(y,k)}^2
+ \abs{m_{i,j}(x,\xi) - m_{i,j}(y,\xi)}^2\,\le\, C_{R}\,\abs{x-y}^2
\end{equation*}
for all $x,y\in B_R$, $i,j,k\in\cS$ and $\xi\in\Act$.
\medskip
\item[\hypertarget{A2}{{(A2)}}]
\emph{Affine growth condition:\/}
$b(x,k,\xi)$ and $\upsigma(x,k)$ satisfy a global growth condition of the form
\begin{equation*}
\sup_{\xi\in\Act}\, \langle b(x,k,\xi),x\rangle^{+} + \norm{\upsigma(x,k)}^{2} \,\le\,C_0
\bigl(1 + \abs{x}^{2}\bigr) \qquad \forall\, k\in\cS, x\in\RR^{d}, 
\end{equation*}
for some constant $C_0>0$.

\medskip
\item[\hypertarget{A3}{{(A3)}}]
\emph{Nondegeneracy:\/}
For each $R>0$, it holds that
\begin{equation*}
\sum_{i,j=1}^{d} a^{ij}(x,k)\zeta_{i}\zeta_{j}
\,\ge\,C^{-1}_{R} \abs{\zeta}^{2} \qquad\forall\, k\in\cS, x\in B_{R}\,,
\end{equation*}
and for all $\zeta=(\zeta_{1},\dotsc,\zeta_{d})\transp\in\RR^{d}$,
where, $a\df \frac{1}{2}\upsigma \upsigma\transp$.
\end{itemize}
The ergodic behavior of $Y_t \,\df \, (X_t,S_t)$ has strong dependency on the coupling coefficients. 
We now define a matrix $\Tilde{M} \df (\Tilde{m}_{i,j}):\RR^d\times\Act\to\RR^{N\times N}$, where
\begin{equation*}
\Tilde{m}_{i,j}(x,\xi)\,\df\,\begin{cases}
m_{i,j}(x,\xi) & \text{if}\,\, i\ne j,\\
0 & \text{otherwise}.
\end{cases}   
\end{equation*}
In addition to the usual structural assumptions
\hyperlink{A1}{{(A1)}}--\hyperlink{A3}{{(A3)}}, we also assume that
\begin{itemize}

\item[\hypertarget{A4}{{(A4)}}]
The matrix $\Breve{M}(x):=(\breve{m}_{ij}(x))$,
where $\breve{m}_{ij}=\min_{\xi\in\Act}\tilde{m}_{ij}(x, \xi)$, is irreducible in $\RR^d$, that is, for every nonempty
$\cS_1, \cS_2\subset \cS$ satisfying $\cS_1\cap\cS_2=\emptyset$ and $\cS_1\cup\cS_2=\cS$ there exist $i_0\in \cS_1$ and $j_0\in\cS_2$ such that
$$|\{x\in\Rd\; :\; \breve{m}_{i_0j_0}(x)>0\}|>0.$$
\end{itemize}

It is well known that under hypotheses
\hyperlink{A1}{{(A1)}}--\hyperlink{A3}{{(A3)}},
\cref{E1.1} has a unique strong solution for
every admissible control (see for example, [\cite{GS72}, Chapter 3]) with $X\in C(\RR_{+};\RR^d), S_{t}\in \mathcal{D}(\RR_{+};\cS)$, where $\mathcal{D}(\RR_{+};\cS)$ is the space of all right continuous functions from $\RR_{+}$ to $\cS$ having left limit. By Markov control we mean an admissible control of the form $v(t,X_t,S_t)$ for some Borel measurable function $v:\RR_+\times\Rd\times\cS\to\Act$. If $v$ is independent of $t$, we call it a stationary Markov control and the set of all stationary Markov controls is denoted by $\Usm$. The hypotheses in \hyperlink{A1}{{(A1)}}--\hyperlink{A3}{{(A3)}} imply the existence of unique
strong solutions under Markov controls (see, \cite[Theorem~2.1]{AGM93}).

Let $c\colon\Rd\times\cS\times\Act\to\RR_+$ be a continuous function that is locally Lipschitz in $x$ uniformly with respect to $\xi\in\Act$, which represents the \emph{running cost}.
Given any admissible control $Z\in\Uadm$, we define the \emph{risk-sensitive criterion} by
\begin{equation*}
\sE_{x,k}(c, Z) \,\df\,
\limsup_{T\to\infty} \, \frac{1}{T}\,
\log \Exp_{x,k}^{Z}\left[e^{\int_0^T c(X_s, S_s, Z_s) \D s}\right],\end{equation*}
and the \emph{optimal value} is defined by
\begin{equation}\label{Eoptval}
\sE^*(c) \,\df\, \inf_{x\in\Rd, i\in\cS}\, \inf_{Z\in\Uadm}\, \sE_{x,i}(c,Z)\,.
\end{equation} 
For $f\in \cC^2(\Rd\times\cS)$, we define the operator $\Lg$ by
\begin{equation*}
(\Lg f)_k(x,\xi) \,\df\, \trace\bigl(a_{k}(x)\grad^2 f_{k}(x)\bigr) + b_{k}(x,\xi)\cdot \grad f_{k}(x) + \sum_{j=1}^{N} m_{k,j}(x,\xi)f_{j}(x), \quad \forall \,\, k\in\cS.
\end{equation*}
We also define the operator $\cA$ mapping
$\cC^2(\Rd\times\cS)$ to $\cC(\Rd\times\cS)$ by
\begin{equation}\label{EcA}
(\cA f)_k(x) \,\df\, \trace(a_{k}(x)\grad^2 f_{k}(x)) + \inf_{\xi\in\Act}\{b_{k}(x,\xi)\cdot\grad f_{k}(x) + c_{k}(x,\xi)f_{k}(x) + \sum_{j=1}^{N} m_{k,j}(x,\xi)f_{j}(x)\}.
\end{equation}
Also, for any stationary Markov control $v$, we define
\begin{equation}\label{EcM}
(\cA^v f)_k(x) \,\df\, \trace(a_{k}(x)\grad^2 f_{k}(x)) + b_{k}(x,v(x))\cdot\grad f_{k}(x) + c_{k}(x,v(x))f_{k}(x) + \sum_{j=1}^{N} m_{k,j}(x,v(x))f_{j}(x).
\end{equation}

In order to study the minimization problem we impose the following geometric stability condition on the dynamics.

\begin{assumption}\label{A1.1}
There exists an inf-compact function $\ell\in\cC(\RR\times\cS)$ (i.e., for each $\hat{k}\in\RR$ the sub-level set $\{\ell \leq \hat{k}\}$ is either empty or compact), and a positive function $\Lyap \in \cC^2(\Rd\times\cS)$, $\Lyap\ge 1$ such that
\begin{equation*}
(\Lg \Lyap)_k(x,\xi) \,\le\, \beta \Ind_{\cK}(x)-\ell_{k}(x) \Lyap_{k}(x)\quad \forall \,\, k\in\cS, (x,\xi)\in \RR^d\times\Act
\end{equation*}
for some constant $\beta$ and compact set $\cK$.
Moreover, we have
\begin{equation*}
\ell(\cdot) - \sup_{\xi\in\Act} c_{k}(\cdot,\xi) \quad \text{is inf-compact}\quad \forall \,\,\, k\in\cS\,.
\end{equation*}
\end{assumption}

It is well known that, if $a$ and $b$ are bounded,
it might not be possible to find an unbounded function $\ell$
satisfying \cref{A1.1} (cf. \cite{ABS19}).
So, in such a case, we instead assume the following.

\begin{assumption}\label{A1.2}
There exists a positive $\Lyap \in \cC^2(\Rd\times\cS)$, $\Lyap\ge 1$, such that
\begin{equation*}
(\Lg \Lyap)_k(x,\xi) \,\le\, \beta \Ind_{\cK}(x) - \gamma \Lyap_k(x)
\end{equation*}
for some positive constants $\beta$ and $\gamma$, and compact set $\cK$.
Moreover, the cost function $c$ is bounded and
\begin{equation*}
\norm{c_{k}}_{\infty} < \gamma \quad \forall \,\, k\in\cS.
\end{equation*}
\end{assumption}

Before we proceed further, let us
exhibit a class of dynamics satisfying \cref{A1.1} or \cref{A1.2}

\begin{example}
Suppose that
$\max_{k\in\cS}\sup_{\xi\in\Act} b_{k}(x,\xi)\cdot x\le - \kappa \abs{x}^\alpha$
outside a compact set for some $\alpha\in (1, 2]$, and $\max_{k\in\cS}a_{k}$ is bounded. Taking $\Lyap_{k} (x)\,\df\,\exp(\theta \sqrt{\abs{x}^2+1})$ for all $k\in\cS$, it follows that
\begin{equation*}
\Lg_{k} \Lyap(x,\xi) \,\le\, \kappa_1 \biggl(\Ind_{\cK_1}(x)
+ \theta\frac{1}{\sqrt{\abs{x}^2+1}}
+ \theta^2 \frac{\abs{x}^2}{\abs{x}^2+1}\biggr)\Lyap_{k}(x) - \theta\frac{\abs{x}^{\alpha}}{\sqrt{\abs{x}^2+1}}\Lyap_{k}(x)
\end{equation*}
for some constant $\kappa_1$, and a compact set $\cK_1$.
Now for $\alpha=1$, if we choose $\theta$ to be sufficiently small we obtain \cref{A1.2}. If $\alpha>1$, we choose $\ell\sim\abs{x}^{\alpha-1}$, which satisfies \cref{A1.1}.  
\end{example}

Let us now state one of our main result of this article.

\begin{theorem}\label{T1.1}
Grant either \cref{A1.1} or \cref{A1.2}.
Then we have the following.
\begin{itemize}
\item[(a)]
There exists a positive $\psi\in \cC^2(\Rd\times\cS)\cap\order(\Lyap)$, satisfying
\begin{equation}\label{ET1.1A}
(\cA \psi)_k(x) \,=\, \lambda^* \, \psi_{k}(x)\ \ \text{ in\ }\Rd\,, \quad\forall\,\, k\in\cS.
\end{equation}
Let  $\bUsm$ denote the class of stationary Markov control
$v$ satisfying 
\begin{align*}
\inf_{\xi\in\Act}\, \bigl\{b_{k}(x,\xi)\cdot \grad \psi_{k}(x) + & c_{k}(x,\xi) \psi_{k}(x) + \sum_{j=1}^{N} m_{i,j}(x,\xi)\psi_{j}(x)\bigl\} \,=\, b_{k}(x,v(x,k))\cdot\grad \psi_{k}(x) \\ 
& + c_{k}(x,v(x,k)) \psi_{k}(x) + \sum_{j=1}^{N} m_{i,j}(x,v(x,k))\psi_{j}(x) \quad \text{a.e.\ in\ } x\in\Rd,\, k\in\cS\,.
\end{align*}
\item[(b)] Any member of $\bUsm$ is an optimal control.
\item[(c)]
There exists a unique (after normalization) positive
$\psi\in\cC^2(\Rd\times\cS)\cap\order(\Lyap)$ satisfying \cref{ET1.1A}.
\item[(d)] Every optimal $v\in\Usm^{*}$ belongs to  $\bUsm$, where $\Usm^{*}$ is the collection of all optimal stationary Markov controls.
\end{itemize}
\end{theorem}

\section{Dirichlet eigenvalue problems in bounded domains}\label{S-bounded}
In this section we consider the principal eigenvalue problem in the bounded domain which we will need later.
Let $D$ be a bounded smooth domain in $\Rd$. Without any loss of generality we may assume that $0\in D$.
Let us point out that  compactness of $\Act$ and non-negativity of $c$ are not required in this section. For this section the only hypothesis we require is stated as follows:
\begin{assumption}
The following hold.
\begin{itemize} 
\item[(i)] $a\in\cC(\bar{D}\times\cS)$ and $\Lambda^{-1} I \le a(x,k) \le \Lambda I$, $x\in\bar D$, $k\in\cS$ for some $\Lambda > 0$.
\item[(ii)]
$b:D\times\cS\times\Act\to\Rd$, $c:D\times\cS\times\Act\to\RR$ and $m_{ij}: D\times\Act\to \RR$ are continuous and bounded.
\item[(iii)] $\Breve{M}(x)$ is irreducible in $D$.
\end{itemize}
\end{assumption}

Next, we define the generalized Dirichlet principal eigenvalue $\lambda_D$ of $\cA$ on the domain $D$ (see also \cite[Appendix A]{ABP-21}). For $\lambda\in\RR$, let us first define
\begin{equation*}
\Uppsi_{D}^+(\lambda) \,\df\, \bigl\{\tilde{\psi}\in \Sobl^{2,d}(D\times\cS)\cap \cC(\bar{D}\times\cS)\,\colon\, \tilde{\psi}\gg 0 \text{\ in\ } D\times\cS\,,\ (\cA\tilde{\psi})_k(x) \,\le \,\lambda\tilde{\psi}_{k}(x) \text{\ a.e. in\ }  D, \quad \forall \,  k\in\cS \bigr\},
\end{equation*}
and the generalized Dirichlet principal eigenvalue in $D$ is defined as
\begin{equation}\label{E2.1}
\lambda_D \,\df\, \inf\,\bigl\{\lambda\in\RR\,\colon\,
\Uppsi_{D}^+(\lambda)\ne \varnothing\bigr\}\,.
\end{equation}
In \cref{T2.3} below we prove existence of a principal eigenfunction corresponding to $\lambda_D$. Toward this goal, we need the
following existence result.

\begin{theorem}\label{T2.1}
Assume that $c\le 0$. Then for any $f\in \cC(\bar D\times\cS)$, there exists a unique solution $w\in\cC_0(\bar{D}\times\cS)\cap\Sob^{2,p}(D\times\cS),\, p\ge d$, satisfying
\begin{equation*}
(\cA w)_k(x) \,=\, f_{k}(x) \quad \text{in\ } D\,, \quad w_{k} = 0\quad \text{on\ } \partial{D}\quad\forall \,\, k\in\cS\,.
\end{equation*}
\end{theorem}

\begin{proof}
For $g\in \cC_0(\bar{D}\times\cS)\cap \cC^1(\bar{D}\times S) (\,\df\,\cC_0^1(D\times\cS))$, let
\begin{equation*}
\Hg_{k} (g,f)(x)\,\df\, -\inf_{\xi\in\Act}\,\bigl\{b_{k}(x,\xi)\cdot \grad g_{k}(x) + c_{k}(x,\xi) g_{k}(x) + \sum_{j=1}^{N}m_{k,j}(x,\xi)g_{j}(x)\bigr\} + f_{k}(x)\,,\quad\forall \,\, k\in\cS.
\end{equation*} It is well known that \cite[Theorem~9.15]{GilTru}
there exists a unique solution
$w\in \cC_0(\bar{D}\times\cS)\cap\Sob^{2,p}(D\times\cS)$, $p>d$,  satisfying 
\begin{equation}\label{PT2.1A}
\trace (a_{k}\grad^2 w_{k})(x) \,=\, \Hg_{k} (g,f)(x) \quad \text{in}\, D,\quad \text{with}\,\, w_{k} = 0 \,\,\text{on}\,\, \partial{D},\quad \forall\,\, k\in\cS.
\end{equation}
Moreover, we have the following estimate  \cite[Theorem~9.14]{GilTru}
\begin{equation}\label{ET2.1A}
\norm{w}_{\Sob^{2,p}(D\times\cS)}
\,\le\, \kappa\bigl(\norm{w}_{L^\infty(D\times\cS)} + \norm{\Hg(g, f)}_{L^p(D\times\cS)}\bigr)\,,
\end{equation}
for some positive constant $\kappa=\kappa(p, D)$ which independent of $g$, $w$, and $f$.
Applying maximum principle \cite[Theorem~9.1]{GilTru} we deduce that 
\begin{equation*}
\norm{w}_{L^\infty(D\times\cS)} \,\le\, \kappa_1 \norm{\Hg(g,f)}_{L^d(D\times\cS)}
\end{equation*}
for some constant $\kappa_1$ which dependents on $a$ and $D$ . Thus from \cref{ET2.1A}, it follows that
\begin{equation}\label{ET2.1B}
\norm{w}_{\Sob^{2,p}(D\times\cS)} \,\le\, \kappa_2 \norm{\Hg(g,f)}_{L^{p}(D\times\cS)}
\end{equation}
for some constant $\kappa_2$. Now,
 let $\cG:\cC_0^1(D\times\cS)\to \cC_0(\bar{D}\times\cS)\cap\Sob^{2,p}(D\times\cS)$, $p>d$, denote the operator mapping $g\in \cC_0^1(D\times\cS)$ to the corresponding solution $w\in\cC_0(\bar{D}\times\cS)\cap\Sob^{2,p}(D\times\cS)$, $p>d$ (see \eqref{PT2.1A}).
Since $p> d$, by standard Sobolev embedding theorem, we have $\Sob^{2,p}(D\times\cS)\hookrightarrow\cC^{1,\alpha}(\bar{D}\times\cS)$ is compact for some $ 0 < \alpha < (1 - \frac{d}{p})$. Thus, from \cref{ET2.1B}, 
it is easily seen that $\cG$ is a compact and continuous operator. Next, we claim that the set
\begin{equation*}
\bigl\{g\in \cC_0^1(D\times\cS)\,\colon\, g=\nu\, \cG g \text{\ for some\ } \nu\in [0, 1]\bigr\}
\end{equation*}
is bounded in $\cC_0^1(D\times\cS)$.
If the claim is not true, then there must exists a sequence $(w_n, \nu_n)$ with $\norm{w_n}_{\cC_0^1(D\times\cS)}\to \infty$
and $\nu_n\to\nu\in[0,1]$ as $n\to\infty$.
We scale the solutions appropriately so that $\norm{w_n}_{\cC_0^1(D\times\cS)} = 1$. Thus, applying \cref{ET2.1B}, we extract a  subsequence of $\{w_n\}$ that converges to a nontrivial solution $w\in \Sobl^{2,p}(D)$ of
\begin{equation*}
\trace (a_{k}\grad^2 w_{k})(x) \,=\, - \nu\inf_{\xi\in\Act}\,\bigl\{b_{k}(x,\xi)\cdot \grad w_{k}(x) + c_{k}(x,\xi) w_{k}(x)+ \sum_{j=1}^{N}m_{k,j}(x,\xi)w_{j}(x)\bigr\} \quad \text{in}\, D,\end{equation*} 
with $w_{k} = 0$ on $\partial{D}$, for all $k\in\cS,$ for some $\nu\in [0,1]$.
But this contradicts to the ABP maximum principle (see \cite[Theorem 1]{Sirakov}). The result of \cite{Sirakov} is applicable since we can linearize the above equation by choosing a measurable selector 
(cf. \cite[Theorem~18.17]{AliBor}).
 This proves that the claim is true. 
Therefore, by the Schauder fixed point theorem, there exists a fixed point $w\in\cC_0(\bar{D}\times\cS)\cap\Sob^{2,p}(D\times\cS)$, $p > d$, of $\cG$. This proves the existence of a solution and the uniqueness follows from the ABP estimate (\cite[Theorem 1]{Sirakov}).
\end{proof}


Let us now recall the nonlinear Krein--Rutman theorem from \cite{A18} which will be used to prove the existence of a principal eigenvalue.

\begin{theorem}\label{T2.2}
Let $\sP$ be a nonempty cone in an ordered Banach space $\cX$.
Suppose that $\cT\colon\cX\to\cX$ is order-preserving, $1$-homogeneous,
completely continuous map and
for some nonzero $u$, and $M>0$ we have $u\preceq M \cT u$.
Then there exists $\lambda>0$ and $x\ne 0$ in $\sP$ such that $\cT x=\lambda x$.
\end{theorem}
The ordering in the above theorem is defined with respect to the cone $\sP$, that is,
for any $u_1, u_2 \in \cX$ we write $u_1\preceq u_2$ if $u_2- u_1 \in \sP$\,. 

\begin{theorem}\label{T2.3}
There exists a unique $\varphi\in \cC_{0}(\bar{D}\times\cS)\cap\Sobl^{2,p}(D\times\cS)$, $p>d$,
satisfying
\begin{align*}
(\cA \varphi)_{k} (x) &\,=\, \lambda_D\, \varphi_k(x) \quad \mbox{in\ } D\,,
\\
\varphi_k &\,=\, 0\quad \text{on\ } \partial{D}\,,
\\
\varphi_k &\, > \, 0 \quad \text{in\ } D\,, \quad \varphi_k(0)=1\,,\quad\forall\quad k\in\cS.
\end{align*}
\end{theorem}
\begin{proof} Let us first assume that $c\le 0$. Also, let $\cX=\cC_0(D\times\cS)$ and $\sP$ be the cone of non-negative functions.
For given $u\in\cC_0(D\times\cS)$, we denote $\cT u = w \in \cC_0(\bar{D}\times\cS)\cap\Sob^{2,p}(D\times\cS)$ to be the solution of
\begin{equation}\label{ET2.3A}
(\cA w)_k(x) \,=\, - u_{k}(x)\quad \text{in\ } D\,, \quad \text{and\ } w_{k} = 0 \text{\ on\ } \partial{D}\,,\quad \forall\quad k\in\cS.
\end{equation}
In view of \cref{T2.1} it is clear that the above map is well defined.
Now by ABP maximum principle \cite[Theorem 1]{Sirakov}, we obtain
\begin{equation}\label{ET2.3ABP}
\sup_{D}\, \max_{k\in\cS}\, \abs{w_{k}} \,\le\, \kappa\, \sup_{D}\, \max_{k\in\cS}\, \abs{u_{k}}
\end{equation}
for some constant $\kappa$, which does not depent on $w$, or $u$.
Next, For some minimizing selector $v:\RR^d\times\cS\to\Act$ of \cref{ET2.3A} it follows that, for all $k\in\cS$
\begin{align*}
\trace\bigl(a_{k}(x)\grad^2 w_{k}(x)\bigr) + b_{k}(x,v(x,k))\cdot \grad w_{k}(x) + (c_{k}(x,v(x,k)) & + m_{k,k}(x,v(x,k)))w_{k}(x) \\
& \,=\, - u_k - \sum_{j\ne k} m_{k,j}(x,v(x,j))w_{j}(x)\,.
\end{align*}
Thus, by standard Sobolev estimate \cite[Theorem~9.14]{GilTru} and \cref{ET2.3ABP}, we deduce that
\begin{equation*}
\norm{w}_{\Sob^{2,p}(D\times\cS)} \,\le\, \kappa_1 \sup_{D}\,\max_{k\in\cS}\, \abs{u_{k}}
\end{equation*}
for some positive constant $\kappa_1$.
Above estimate implies that $\cT $ is a continuous and compact operator. From the definition of $\cT$ it is easy to see that $\cT$ is $1$-homogeneous, that is, $\cT(\hat{\lambda} u) = \hat{\lambda} \cT(u)$ for all $\hat{\lambda} \geq 0$\,. Also, let $\hat{u}_i\in\cC_0(D\times\cS)$ such that $\cT(\hat{u}_i) = \hat{w}_i$\,, i=1,2 and $\hat{u}_1 \leq \hat{u}_2$. Thus, we have 
\begin{equation*}
(\cA \hat{w}_1)_k(x) \,\ge\, (\cA \hat{w}_2)_k(x),\quad\forall \,\,k\in\cS.
\end{equation*}
The concavity of $\cA$ implies that
\begin{equation}\label{ET2.3B}
(\cA (\hat{w}_2 - \hat{w}_1))_k(x) \,\le\, 0\, \text{in}\,\, D,\,\, \forall\,\ k\in\cS.
\end{equation}
Thus, we see from \cite[Theorem 1]{Sirakov} that $\hat{w}_2\ge \hat{w}_1$. This in particular implies $\cT(\sP)\subset\sP$. In fact, when the inequality is strict, that is,  $\hat{u}_1\lneq \hat{u}_2$, then $\hat{u}_{1,k} < \hat{u}_{2,k}$ for some $k\in\cS$. Letting $\tilde{w} = \hat{w}_2 - \hat{w}_1$, we see that 
\begin{equation*}
\trace (a_{k}\grad^2 \tilde{w}_{k})(x) - M\abs{\grad \tilde{w}_{k}} - \norm{c_{k} + m_{k,k}}_\infty \tilde{w}_{k}(x)
\,\le\, (\cA \tilde{w})_k(x) \,\le\, (\hat{u}_{1,k} - \hat{u}_{2,k})(x) \lneq 0\,,
\end{equation*} 
where $M$ is a positive constant such that $M \,\geq\, \sup_{D\times\Act}\max_{k\in\cS}\abs{b_{k}(x,\xi)}$.
Then by a version of Hopf's boundary lemma (see for example, \cite[Lemma~3.1]{QS08}), we must have $\tilde{w}_{k} = \hat{w}_{2,k} - \hat{w}_{1,k} > 0$ in $D$. Moreover, since the system is irreducible we claim that $\tilde{w}_{k} > 0$ for all $k\in\cS$. More precisely, due to the irreducibility property of the system it is clear that for each $k\in\cS$ there exist $j\in\cS$ such that $m_{j,k} \gneq 0$. We know that $\tilde{w}_{j}\geq 0$, but if $\tilde{w}_{j} \equiv 0$ it contradicts the fact that $\cA_{j}\tilde{w} \leq 0$ (since $\tilde{w}_{k} > 0$). Continuing this process by irreducibility property of the system we obtain the claim.
 Therefore  $\hat{u}_1\lneq \hat{u}_2$ implies $\cT(\hat{u}_{2})\gg \cT(\hat{u}_{1})$.

Now consider a function $u\in \sP$ which has compact support in $D$ and $u \neq 0$. Then, it follows from above discussion that $w = \cT u \gg 0$ in $D\times\cS$.
Thus, we can find $M>0$ satisfying $M\cT u - u \gg 0$ in $D\times\cS$. Therefore, by the Krein--Rutman theorem (see \cref{T2.2} ),  there exist $\lambda>0$ and $ \phi \gg 0$ in $\cC_0(\bar{D}\times\cS)\cap\Sob^{2,p}(D\times\cS)$ such that
\begin{equation*}
(\cA \phi)_{k} \,=\, \lambda \phi_{k}\quad \text{in\ } D\,,\quad \forall\,\,k\in\cS.
\end{equation*}
Now, as we know that $c$ is bounded, replacing $c$ by $c -2\norm{c}_{L^\infty( D\times\cS)}$, it follows from the above discussion that there 
exists $\lambda\in\RR$ and $\varphi\in \cC_0(\bar{D}\times\cS)\cap\Sobl^{2,p}(D\times\cS)$, $p>d$, satisfying
\begin{equation*}
(\cA \varphi)_{k} \,=\, \lambda\, \varphi_{k} \quad \text{in\ } D\,, \qquad \varphi_{k} > 0 \quad \mbox{in\ } D\,,\qquad \text{and}\,\, \varphi_{k} \,=\,0\quad \text{on}\,\, \partial{D}\,,\quad\forall \,\, k\in\cS. 
\end{equation*}

 Next, we want to prove that if there exists pair
$(\tilde{\lambda}, w)\in \RR\times \cC(\bar{D}\times\cS)\cap\Sobl^{2,d}(D\times\cS)$, with $ w \gg 0$ in $D\times\cS$, such that
\begin{equation*}
(\cA w)_{k}(x) \le  \tilde{\lambda} w_{k}(x) \,\,\, \text{in} \,\, D, \quad\forall \,\,\, k\in\cS.
\end{equation*} Then either $\tilde{\lambda} > \lambda$ or $\tilde{\lambda} = \lambda$, $w = t\varphi$ for some positive constant $t$.

It is enough to show that if $ w\in \cC(\bar{D}\times\cS)\cap\Sobl^{2,p}(D\times\cS)$ with $ w \gg 0$ in $D\times\cS$, and satisfies
\begin{equation*}
(\cA w)_{k}(x) \le  \lambda w_{k}(x) \,\,\, \text{in}\,\,\, D, \quad \forall \,\,\, k\in\cS.
\end{equation*} 
Then $w = t\varphi$ for some constant $t > 0$.

Define $u_{t} \,\df \, t\varphi - w$. Then for some suitable choice of $t > 0$ (small enough), we have $u_{t} \le 0$ in some compact set $K \subsetneq D$ with $|D\setminus K| < \epsilon$, where $\epsilon > 0$ is small number. Also, we have the following 
\begin{equation*}
\trace (a_{k}\grad^2 u_{t,k})(x) + \sup_{\xi\in\Act}\,
\bigl\{b_{k}(x,\xi)\cdot \grad u_{t,k}(x)
+ (c_{k}(x,\xi)-\lambda) u_{t,k}(x) + \sum_{j=1}^{N}m_{k,j}(x,\xi)u_{t,j}\bigr\} \,\ge\,  0\,.
\end{equation*}
For any maximizing selector of the above equation, choosing $\varepsilon$ sufficiently small, applying \cite[Theorem 1]{Sirakov} corresponding to the domain $D\setminus K$, we see that $u_t\le 0$ in $D\times\cS$. Thus, $u_t\le 0$ in $D\times\cS$ satisfies
\begin{equation*}
\trace (a_{k}\grad^2 u_{t,k})(x) + \sup_{\xi\in\Act}\,
\bigl\{b_{k}(x,\xi)\cdot \grad u_{t,k}(x)
- (c_{k}(x,\xi) + m_{k,k}(x,\xi) - \lambda)^{-} u_{t,k}(x)\bigr\} \,\ge\,  0\,.
\end{equation*}
Then by strong maximum principle \cite[Theorem~9.6]{GilTru},
we must either have $u_{t,k}=0$ or $u_{t,k}<0$ in $D$. By the previous discussion, the cooperative nature of the system implies that if $u_{t,k} < 0$ for some $k\in\cS$ then $u_{t,j} < 0$ for all $j\in\cS$. Thus we either have $u_{t}=0$ or $u_{t} \ll 0$ in $D\times\cS$.
Suppose that the second option holds. Then we may define
\begin{equation*}
\mathfrak{t} \,=\, \sup\,\{t>0 \,\colon\, u_t \ll 0 \quad \text{in\ } D\times\cS\}\,.
\end{equation*}
By the above argument, $\mathfrak{t}>0$, and by strong maximum principle \cite[Theorem~9.6]{GilTru} we must have either $u_\mathfrak{t}=0$ or $u_\mathfrak{t}\ll 0$.
If $u_\mathfrak{t} < 0$, then for some $\delta>0$ we have $u_\mathfrak{t+\delta} \ll 0$ in $K\times\cS$,
and therefore, repeating the argument above, we obtain $u_\mathfrak{t  + \delta} \ll 0$ in $D\times\cS$.
This contradicts the definition of $\mathfrak{t}$. So the only possibility is $u_\mathfrak{t} = 0$. This indeed implies that $\lambda = \lambda_{D}$. This completes the prove.
\end{proof} 
Next theorem shows that the Dirichlet principal eigenvalue $\lambda_D$ is monotone with respect to the potential term $c$\,.
\begin{theorem}\label{T2.4}
Let $c\lneq c'$. Suppose that $\lambda_D(c)$ \textup{(}$\lambda_D(c')$\textup{)}
is the Dirichlet principal eigenvalue. Then $\lambda_D(c)<\lambda_D(c')$.
\end{theorem}
\begin{proof}
Let $\varphi_{c}, \varphi_{c'}$ denote the principal eigenfunctions corresponding to $c, c',$ respectively. It is clear from \cref{E2.1} that $\lambda_D(c) \le \lambda_D(c')$. Suppose that $\lambda_D(c) \,=\, \lambda_D(c')$. Then, we obtain 
\begin{equation*}
(\cA \varphi_{c'})_{k}(x) \le  \lambda_{D}(c) \varphi_{c',k}(x) \,\,\, \text{in}\,\,\, D, \quad \forall \,\,\, k\in\cS.
\end{equation*}Now, from the proof of \cref{T2.3}, it follows that $\varphi_{c'} = t\varphi_{c}$ for some positive constant $t$. But this contradict the fact that $c\lneq c'$. Therefore, we have $\lambda_D(c) < \lambda_D(c').$
\end{proof}
We also have the following monotonicity property of the Dirichlet principal eigenvalue $\lambda_{D}$ with respect to the domains.
\begin{theorem}\label{T2.5}
Let $D_{1}\subsetneq D_{2}$. Then $\lambda_{D_1} < \lambda_{D_2}$.
\end{theorem}
\begin{proof}
Let $\varphi_{D_{1}}, \varphi_{D_{2}}$ denote the principal eigenfunctions corresponding $\lambda_{D_1}, \lambda_{D_2}$ respectively. From the definition \cref{E2.1}, it follows that $\lambda_{D_1} \le \lambda_{D_2}$. If $\lambda_{D_1} = \lambda_{D_2}$, then 
\begin{equation*}
(\cA \varphi_{D_{2}})_{k}(x) \le  \lambda_{D_{2}} \varphi_{D_{2},k}(x) \,\,\, \text{in}\,\,\, D_1, \quad \forall \,\,\, k\in\cS.
\end{equation*} As in the proof of \cref{T2.3}, this implies $\varphi_{D_{2}} = t \varphi_{D_{1}}$ for some $t > 0$. This contradict the fact that $\varphi_{D_{2}} \gg 0$ in $D_{2}\times\cS$, since $D_{1}\subsetneq D_{2}$ and $\varphi_{D_{1}} = 0$ on $\partial{D}_{1}\times\cS$.  Thus, we obtain $\lambda_{D_1} < \lambda_{D_2}$. 
\end{proof}
Let $\{D_n\}_{n\in\NN}$ be a decreasing sequence of smooth domains whose intersection is $D$, and which satisfies an exterior sphere condition uniformly in $n\in\NN$,
that is, there exists $r>0$ such that for all large $n$,
every point of $\partial D_n$ can be touched from outside of $D_n$ with a ball of radius $r$. 
 
Next we address the continuity properties of the principal eigenvalue with respect to the domain $D$. In order to proof the continuity property we need the following boundary estimate. For the proof, see \cite[Lemma~6.1]{AB-19}.
\begin{lemma}\label{L2.1}
Suppose that $\norm{w}_{\infty;D\times\cS} \le 1$, and it satisfies
\begin{equation*}
\trace (a_{k}\grad^2 w_{k}) + \delta\abs{\grad w_{k}} \,\ge\, L \text{\ \ in\ } D\,, \quad w = 0
\text{\ \ on\ } \partial{D}\times\cS\,\quad \forall \,\,\, k\in\cS
\end{equation*}
where $D$ has an exterior sphere property of radius $r>0$.
Then for $s\in(0,1)$, there exist constants $M$, and $\varepsilon$,
depending only on
$\delta$, $L$, $r$, and $s$, such that
\begin{equation*}
\max_{k\in\cS}\abs{w_{k}(x)}\,\le\, M \dist(x,\partial D)^s, \quad \text{for all $x$ such that\ }
\dist(x, \partial D)<\varepsilon\,.
\end{equation*}
\end{lemma} 
 
\begin{theorem}\label{T2.6}
Suppose that $D_n\to D$ as above. Then $\lambda_{D_n}\to \lambda_D$, as $n\to\infty$.
\end{theorem}
\begin{proof}
From \cref{T2.5} it is clear that $\lambda_{D_n}$ is a decreasing sequence and bounded below by $\lambda_D$. Thus, $\lambda_{D_n}$ converges to a number $\Tilde\lambda$, with $\Tilde\lambda \ge \lambda_D$.
We normalize the eigenfunctions so that $\norm{\varphi_{D_n}}_{\infty; D_{n}\times\cS} = 1$. Now, using \cref{L2.1} and the interior estimate, it can be easily seen that the family
$\{\varphi_{D_n}\}$ is equicontinuous and each limit point $ \phi\in \cC(\bar{D}\times\cS)\cap\Sobl^{2,p}(D\times\cS)$ is a nonnegative solution to
\begin{equation*}
(\cA \phi)_{k}(x) =  \Tilde{\lambda} \phi_{k}(x) \,\,\, \text{in}\,\,\, D, \quad \forall \,\,\, k\in\cS.
\end{equation*}
By the strong maximum principle, we must have $\phi \gg 0$ in $D\times\cS$. Then it follows from the
proof of \cref{T2.3} that $\Tilde\lambda=\lambda_D$.
\end{proof}
%
\section{Ergodic Risk-sensitive controls}\label{S-risk}
In this section, we study the risk-sensitive control problem, and characterize
the corresponding eigenfunction $V$.
Assumptions \hyperlink{A1}{{(A1)}}--\hyperlink{A4}{{(A4)}}
are in full effect in this section.
In view of \cref{T2.3,T2.5}, we have the following result on the existence of Dirichlet principal eigenpair on smooth bounded domains $\sB_n$, $n\in\NN$.
\begin{lemma}\label{L3.1}
For each $n\in\NN$ there exist $(\psi_n,\lambda_n)\in\cC_0(\bar{\sB}_n\times\cS)\cap\Sob^{2,p}(\sB_n\times\cS)\times\RR$, $p>d$, the Dirichlet principal eigenpair satisfying
\begin{equation}
\begin{split}\label{EL3.1A}
(\cA\psi_n)_{k}  & = \lambda_n\, \psi_{n,k} \quad \text{in\ } \sB_n\,,\\
\psi_{n,k} & = 0 \quad \text{in\ } \partial{\sB}_n\,,\\
\psi_{n,k} & > 0 \quad \text{in\ } \sB_n,  \quad \forall \,\,\, k\in\cS.
\end{split}
\end{equation}
 Moreover, we have the following
\begin{itemize}
\item[(a)] $\lambda_{n} < \lambda_{n+1}$, for all $n\in\NN$.
\item[(b)]For every minimizing selector $v_{n}^{*}$ of \cref{EL3.1A} and $r\in (0, n)$, we have
\begin{equation}\label{EL3.1B}
\psi_{n,k}(x) \,=\, \Exp_{x,k}^{v_{n}^{*}}\left[ e^{\int_0^{\uuptau_r} (c(X_s,S_s,v_{n}^{*}(X_s,S_s))-\lambda_n) \D{s}}\, \psi_n(X_{\uuptau_r}, S_{\uuptau_r})\Ind_{\{\uuptau_r<\uptau_n\}}\right]
\quad \forall\,x\in \sB_n\setminus \overline{\sB_r}\,,
\end{equation}
where $\uuptau_r=\uptau(B^c_r)$ denotes the hitting time of the process $X_{t}$ to the ball $\sB_r$. 
\item[(c)] For all $n\in\NN$, we have $\lambda_n\le \sE_{x,k}(c,Z)$ for $Z\in\Uadm,$ \,\, $x\in\sB_n$,\,\, $k\in\cS$\,.
\end{itemize}
\end{lemma}

\begin{proof}
Existence of principal eigenpair follows from \cref{T2.3}. Part (a) follows from \cref{T2.5}.
For part (b), let $v_{n}^{*}$ be a minimizing selector of \cref{EL3.1A}, we extend it in $\Rd$ by setting $v_{n}^{*} = \xi$ for some fixed $\xi\in\Act$. Then applying It\^{o}-Krylov formula and using the fact $\psi_n=0$ on $\partial{\sB}_n$, we obtain
\begin{equation}\label{EL3.1C}
\psi_{n,k}(x) \,=\, \Exp_{x,k}^{v_{n}^{*}}\left[ e^{\int_0^{T\wedge \uuptau_r} (c(X_s,S_s,v_{n}^{*}(X_s,S_s))-\lambda_n) \D{s}}\, \psi_n(X_{T\wedge \uuptau_r}, S_{T\wedge \uuptau_r})\Ind_{\{\uuptau_r\wedge T <\uptau_n\}}\right],
\quad T\ge 0\,.
\end{equation}
Letting $T\to\infty$ and applying Fatou's lemma it follows from \cref{EL3.1C} that 
\begin{equation}\label{EL3.1D}
\psi_{n,k}(x) \,\geq\, \Exp_{x,k}^{v_{n}^{*}}\left[ e^{\int_0^{\uuptau_r} (c(X_s,S_s,v_{n}^{*}(X_s,S_s))-\lambda_n) \D{s}}\, \psi_n(X_{\uuptau_r}, S_{\uuptau_r})\Ind_{\{\uuptau_r <\uptau_n\}}\right].
\end{equation}
Let  $\Tilde{c}(x) \,\df\,   c(x, v^*_n) - \tilde{\delta}\Ind_{\sB_r}(x)$, for some $\tilde{\delta} > 0$, and $(\Tilde\psi_n, \Tilde\lambda_n)$ be the corresponding principal Dirichlet
eigenpair on $\sB_{n}$ of the linear operator $\Tilde\cA^{v^*_n}$ (see \eqref{EcM}) defined as
\begin{align}\label{EL3.1Linopa}
(\Tilde{\cA}^{v_{n}^{*}} f)_{k}(x) \,\df\, \trace(a_{k}(x)\grad^2 f_{k}(x)) + \{b_{k}(x,v_{n}^{*}(x,k))\cdot\grad f_{k}(x) & + \Tilde{c}_{k}(x,v_{n}^{*}(x,k))f_{k}(x)\nonumber\\
& + \sum_{j=1}^{N} m_{k,j}(x,v_{n}^{*}(x,j))f_{j}(x)\}.
\end{align}
Existence of such eigenpair follows from \cite[Theorem~A.1]{ABP-21}.
Then by strict monotonicity of the Dirichlet principal eigenvalue \cite[Theorem~A.2]{ABP-21}, we have $\lambda_n > \Tilde\lambda_n$ for all $n\in\NN$. By continuity of the Dirichlet principal eigenvalues with respect to the domains \cite[Theorem~A.4]{ABP-21}, we can find a ball
$\sB_{n'}$ containing $\sB_n$ such that $\lambda_n>\Tilde{\lambda}_{n'}$, where
$(\Tilde{\psi}_{n'}, \Tilde{\lambda}_{n'})$ is the principal Dirichlet eigenpair corresponding to $\Tilde{c}$ in $\sB_{n'}$.
Rewriting \cref{EL3.1C}, we obtain
\begin{align}\label{EL3.1E}
\psi_{n,k}(x) \,= &\, \Exp_{x,k}^{v_{n}^{*}}\left[ e^{\int_0^{\uptau_n\wedge \uuptau_r} (c(X_s,S_s,v_{n}^{*}(X_s,S_s))-\lambda_n) \D{s}}\, \psi_n(X_{\uptau_n\wedge \uuptau_r}, S_{\uptau_n\wedge \uuptau_r})\Ind_{\{\uuptau_r\wedge \uptau_n \le T\}}\right] \nonumber\\
& \, + \, \Exp_{x,k}^{v_{n}^{*}}\left[ e^{\int_0^{T} (c(X_s,S_s,v_{n}^{*}(X_s,S_s))-\lambda_n) \D{s}}\, \psi_n(X_{T}, S_{T})\Ind_{\{T < \uuptau_r\wedge \uptau_n\}}\right]
\end{align}
Also, we have
\begin{align*}
\Exp_{x,k}^{v_{n}^{*}}&\left[ e^{\int_0^{T} (c(X_s,S_s,v_{n}^{*}(X_s,S_s))-\lambda_n) \D{s}}\,  \psi_n(X_{T}, S_{T})\Ind_{\{T < \uuptau_r\wedge \uptau_n\}}\right] 
\\
&\,\le\, \Exp_{x,k}^{v_{n}^{*}}\left[ e^{\int_0^{T} (\Tilde{c}(X_s,S_s,v_{n}^{*}(X_s,S_s))-\lambda_n) \D{s}}\, \psi_n(X_{T}, S_{T})\Ind_{\{T < \uuptau_r\wedge \uptau_n\}}\right]
\\
&\,\le\,  \frac{\max_{k\in\cS}\,\sup_{\sB_n}\psi_{n,k}}{\min_{k\in\cS}\,\inf_{\sB_n}\psi_{n',k}}\, \Exp_{x,k}^{v_{n}^{*}}\left[ e^{\int_0^{T} (\Tilde{c}(X_s,S_s,v_{n}^{*}(X_s,S_s))-\lambda_n) \D{s}}\, \psi_{n'}(X_{T}, S_{T})\Ind_{\{T < \uuptau_r\wedge \uptau_n\}}\right]
\\
&\,\le\,  \frac{\max_{k\in\cS}\,\sup_{\sB_n}\psi_{n,k}}{\min_{k\in\cS}\,\inf_{\sB_n}\psi_{n',k}}\,
e^{(\Tilde{\lambda}_{n'}-\lambda_n)T}\Tilde{\psi}_{n',k}(x) \,\xrightarrow[T\to\infty]{}\, 0\,.
\end{align*}
Thus, using the monotone convergence theorem, letting $T\to\infty$ in \cref{EL3.1E}, it follows that
\begin{equation}\label{EL3.1F}
\psi_{n,k}(x) \,\leq\, \Exp_{x,k}^{v_{n}^{*}}\left[ e^{\int_0^{\uuptau_r} (c(X_s,S_s,v_{n}^{*}(X_s,S_s))-\lambda_n) \D{s}}\, \psi_n(X_{\uuptau_r}, S_{\uuptau_r})\Ind_{\{\uuptau_r <\uptau_n\}}\right].
\end{equation}
Now, combining \cref{EL3.1D} and \cref{EL3.1F}, we obtain \cref{EL3.1B}.

For part (c), applying It\^{o}-Krylov formula as in \cref{EL3.1C}, we have
\begin{eqnarray*}
\psi_{n,k}(x) \,& \le &\, \Exp_{x,k}^{Z}\left[ e^{\int_0^{T} (c(X_s,S_s,Z_s))-\lambda_n) \D{s}}\, \psi_n(X_{T}, S_{T})\Ind_{\{ T <\uptau_n\}}\right]\\
&\le & \, \norm{\psi_n}_{L^\infty(\sB_{n}\times\cS)} \Exp_{x,k}^{Z}\left[ e^{\int_0^{T} (c(X_s,S_s,Z_s))-\lambda_n) \D{s}}\right] \quad \forall\,\, T\ge 0, \,\, (x,k)\in\sB\times\cS.
\end{eqnarray*}
Taking logarithm both side, dividing by $T$ and letting $T\to\infty$, we obtain $\lambda_n\le \sE_{x,k}(c,Z)$ for $Z\in\Uadm$,\,\,$x\in\sB_n$,\,\, $k\in\cS$\,.
\end{proof}
We now show that under  \cref{A1.1} or \cref{A1.2}, the optimal value of our optimal problem is finite, that is, $\sE^{*}(c) < \infty$\,. 

\begin{lemma}\label{L3.2}
Suppose that \cref{A1.1} or \cref{A1.2} holds.
Then $\sE^{*}(c) < \infty$. 
\end{lemma}

\begin{proof}
Since under \cref{A1.2}, $\norm{c}_{L^\infty(\Rd\times\cS)}$ is finite, therefore $\sE^{*}(c) \le \norm{c}_{L^\infty(\Rd\times\cS)} < \infty$.

Suppose  \cref{A1.1} holds. Then, we see that
\begin{equation*}
(\Lg\Lyap)_{k} + (\ell_{k} - g_{k})\Lyap_{k} \,\le\, 0 \quad \forall \,\, k\in\cS, (x,\xi)\in \RR^d\times\Act\,, 
\end{equation*} where $g_{k} = \beta [\inf_\cK \Lyap_{k}]^{-1}\Ind_\cK$.
Applying It\^{o}-Krylov formula, for any $Z\in\Uadm$ we obtain
\begin{equation*}
\Lyap_{k}(x) \,\ge\,\Exp_{x,k}^{Z}\left[e^{\int_0^{\uptau_n\wedge T} (\ell(X_s,S_s)-g(X_s,S_s)) \,\D{s}}\,
\Lyap(X_{\uptau_n\wedge T},S_{\uptau_n\wedge T})\right]\,.
\end{equation*}
Letting $n\to\infty$ and applying Fatou's lemma, it follows that
\begin{equation*}
\Lyap_{k}(x) \,\ge\, \Exp_{x,k}^{Z}\left[e^{\int_0^{T} (\ell(X_s,S_s)-g(X_s,S_s)) \,\D{s}}\,
\Lyap(X_{ T},S_{T})\right]
\,\ge\, \bigl(\inf_{\Rd} \,\min_{k\in\cS}\,\Lyap_{k}\bigr)\,
\Exp_{x,k}^{Z}\left[e^{\int_0^{T} (\ell(X_s,S_s)-g(X_s,S_s)) \,\D{s}}\,\right].
\end{equation*}
Taking logarithm on both sides, dividing by $T$,
 and letting $T\to\infty$, we deduce that 
 $$\sE^{*}(\ell)\le \max_{k\in\cS}\beta\, [\inf_\cK \Lyap_{k}]^{-1}.$$
It is also clear from the \cref{A1.1} that $\sup_{\xi\in\Act} c_{k}(x,\xi) \leq \hat{\beta} + l_{k}(x)$, for some constant $\hat{\beta} > 0$, for all $k\in\cS$ and $x\in\Rd$. Thus, we obtain 
$$\sE^{*}(c) < \hat{\beta} + \max_{k\in\cS}\beta\, [\inf_\cK \Lyap_{k}]^{-1}.$$
\end{proof}
Next, we define the generalized principal eigenvalue $\lambda^{*} $ in $\Rd$ of $\cA$. For $\lambda\in\RR$, let us first define
\begin{equation*}
\Uppsi^+(\lambda) \,\df\, \bigl\{\tilde{\psi}\in \Sobl^{2, d}(\Rd\times\cS)\,\colon\, \tilde{\psi}\gg 0 \text{\ in\ } \Rd\times\cS\,,\ (\cA\tilde{\psi})_k(x) \,\le \,\lambda\tilde{\psi}_{k}(x) \text{\ in\ }  \Rd, \quad \forall \,  k\in\cS \bigr\},
\end{equation*}
and the generalized principal eigenvalue is defined as
\begin{equation}\label{E3.8}
\lambda^{*} = \lambda^{*}(c) \,\df\, \inf\,\bigl\{\lambda\in\RR\,\colon\,
\Uppsi^+(\lambda)\ne \varnothing\bigr\}\,.
\end{equation}

In the following lemma we show that for the semi-linear operator $\cA$, there exist positive eigenfunctions in the whole space $\Rd$\,. In particular, we show that as $n\to\infty$ the Dirichlet principal eigenpairs $(\psi_n,\lambda_n)$  on bounded domains $\sB_n$, converge to the principal eigenpair of the semi-linear operator $\cA$ in $\Rd$\,.
\begin{lemma}\label{L3.3} Suppose that \cref{A1.1} or \cref{A1.2} holds. Let $\Tilde{\lambda} = \lim_{n\to\infty}\lambda_n$. Then we have the following:
\begin{itemize}
\item[(a)]
There exists a function
$\psi^*\in \Sobl^{2,p}(\Rd\times\cS)\cap\order(\Lyap)$, $\psi^* \gg 0,$ satisfying
\begin{equation}\label{EL3.3A}
(\cA\psi^*)_{k}(x) \,=\, \Tilde{\lambda} \psi_{k}^*(x)\quad \text{in\ } \Rd\quad \forall \,\, k\in\cS.
\end{equation}

\item[(b)] It holds that $\Tilde{\lambda} = \lamstr$.

\end{itemize}
\end{lemma}

\begin{proof}
From \cref{L3.1}(a) and \cref{L3.2} it is clear that $\Tilde{\lambda}$ exist and $\Tilde{\lambda} < \infty$.
Let $v_{n}^{*}$ be a minimizing selector of \cref{EL3.1A}. Thus, we have
\begin{equation*}
(\cA^{v_{n}^{*}}\psi_{n})_k(x) = \lambda_{n}\psi_{n,k}(x) \quad \text{in}\,\,\sB_{n}\quad\forall\,\,k\in\cS,
\end{equation*} where $\cA^{v_{n}^{*}}$ is defined as 
\begin{align*}
(\cA^{v_{n}^{*}} f)_k(x) \,\df\, \trace(a_{k}(x)\grad^2 f_{k}(x)) + \{b_{k}(x,v_{n}^{*}(x,k))\cdot\grad f_{k}(x) & + c_{k}(x,v_{n}^{*}(x,k))f_{k}(x)\nonumber\\
& + \sum_{j=1}^{N} m_{k,j}(x,v_{n}^{*}(x,j))f_{j}(x)\}.
\end{align*}
Let $\sK\subset\sB_{n}$ be a compact, without loss of generality we assume that $0\in\sK$. Since $\min_{k\in\cS}\psi_{n,k}(0) = 1$ (after normalization), applying Harnack's inequality \cite[Theorem 2]{Sirakov}, it follows that
\begin{equation*}
\sup_{y\in\sK}\,\max_{k\in\cS}\psi_{n,k}(y) \leq \kappa \, ,
\end{equation*}
for some constant $\kappa$ independent of $n$.
Thus, by \cite[Theorem~9.11]{GilTru} we deduce that for any $ Q\subset \sK$  
\begin{equation*}
\norm{\psi_n}_{\Sob^{2,p}(Q\times\cS)} \leq \kappa_{1}, \,\, p > d \,,
\end{equation*} for some $\kappa_{1} > 0$ uniformly in $n$. 
Hence by a standard diagonalization argument we can extract a subsequence $\{\psi_{n_m}\}$ such that 
\begin{equation*}
\psi_{n_m}\to \psi^*\quad \text{in\ } \Sobl^{2,p}(\Rd\times\cS)\,\,\,(\text{weakly}), \quad \text{and}\quad \psi_{n_m}\to \psi^*\quad \text{in\ } \cC^{1, \alpha}_{loc}(\Rd\times\cS)\,\,\,(\text{strongly})
\end{equation*}
for some $\psi^*\in \Sobl^{2,p}(\Rd\times\cS)$.
Moreover, passing the limit in \eqref{EL3.1A}, we have 
\begin{equation}\label{EL3.3C}
(\cA\psi^*)_k(x) \,=\, \Tilde{\lambda}\psi_{k}^*(x)\quad\text{in}\,\, \Rd\quad\forall\,\,k\in\cS\,.
\end{equation}
Since $\min_{k\in\cS}\psi_{k}(0) \ge 1$, by an application of Harnack's inequality \cite[Theorem 2]{Sirakov}, it follows that $\psi^* \gg 0$ in $\Rd\times\cS$. 
From \cref{EL3.3C} it is clear that $ \Tilde{\lambda} \ge \lamstr$. Also, from the definition \cref{E3.8}, we have $\lambda_{n} \le \lamstr$ for all $n\in\NN$. This gives us $ \Tilde{\lambda} \le \lamstr$. Therefore, we obtain $ \Tilde{\lambda} = \lamstr$. This proves (b).

Next we want to prove that $\lamstr \ge 0$. Suppose that $\lamstr < 0$. Then for any minimizing selector $\hat{v}$ of \cref{EL3.3C}, applying It\^{o}-Krylov formula, for any $n > 1$, $T > 0$ we obtain
\begin{equation*}
\psi_{k}^*(x)=\Exp_{x,k}^{\hat{v}}\left[ e^{\int_0^{\uptau_n\wedge\uuptau_{1}\wedge T} (c(X_s,S_s, \hat{v}(X_s,S_s))-\lamstr)\,\D{s}}
\psi^*(X_{\uptau_n\wedge\uuptau_{1}\wedge T}, S_{\uptau_n\wedge\uuptau_{1}\wedge T})\right] \quad \forall\,x\in \sB_{1}^c\cap\sB_{n},\, k\in\cS.
\end{equation*} Since $(c - \lamstr) > 0$, letting $T\to\infty$, and $n\to\infty$ and applying Fatou's lemma, it follows that
\begin{equation*}
\psi_{k}^*(x) \ge \Exp_{x,k}^{\hat{v}}\left[ e^{\int_0^{\uuptau_{1}} (c(X_s,S_s, \hat{v}(X_s,S_s))-\lamstr)\,\D{s}}
\psi^*(X_{\uuptau_{1}}, S_{\uuptau_{1}})\right] \ge \min_{k\in\cS}\,\inf_{y\in\sB_{1}}\psi_{k}^*(y) \quad \forall\,x\in \sB_{1}^c,\, k\in\cS.
\end{equation*}
Using the above estimate, It\^{o}-Krylov formula and Fatous's lemma, we get
\begin{align*}
\psi_{k}^*(x) & \ge \Exp_{x,k}^{\hat{v}}\left[ e^{\int_0^{T} (c(X_s,S_s, \hat{v}(X_s,S_s))-\lamstr)\,\D{s}}
\psi^*(X_{T}, S_{T})\right]\\
& \ge \min_{k\in\cS}\,\inf_{y\in\sB_{1}}\psi_{k}^*(y)\Exp_{x,k}^{\hat{v}}\left[ e^{\int_0^{T} (c(X_s,S_s, \hat{v}(X_s,S_s))-\lamstr)\,\D{s}}\right].
\end{align*} 
Taking logarithm on both sides, dividing by $T$, letting $T\to\infty$, we deduce that
$\lamstr \ge  \sE_{x,k}(c, \hat{v}) \ge 0$. This contradicts the fact that $\lamstr < 0.$ Thus, we obtain $\lamstr \ge 0.$ 
 
Now choose $r_{0}$ large enough so that
$(\sup_{\xi\in\Act}c_{k}(x,\xi)-\lamstr)\le \theta\ell_{k}$ (or $\theta\gamma$) in $\sB_{r_{0}}^c$ and $\cK\subset\sB_{r_{0}}$, for some $\theta\in (0,1)$. Since $\lambda_{n}$ increases to $\lamstr$, for large $n$ we have $(\sup_{\xi\in\Act}c_{k}(x,\xi)-\lambda_{n})\le \theta\ell_{k}$ (or $\theta\gamma$) in $\sB_{r_{0}}^c$.
This is possible due to \cref{A1.1} and \cref{A1.2}. From \cref{L3.1}, for sufficiently large $n$ and $x\in\sB_{r_{0}}^{c}\cap\sB_{n}$, it follows that
\begin{align*}
\psi_{n,k}(x) \,& =\, \Exp_{x,k}^{v_{n}^{*}}\left[ e^{\int_0^{\uuptau_{r_{0}}} (c(X_s,S_s,v_{n}^{*}(X_s,S_s))-\lambda_n) \D{s}}\, \psi_n(X_{\uuptau_{r_{0}}}, S_{\uuptau_{r_{0}}})\Ind_{\{\uuptau_r<\uptau_n\}}\right]
\\
& \le \frac{\max_{k\in\cS}\sup_{y\in\sB_{r_{0}}}\psi_{n,k}(y)}{\min_{k\in\cS}\inf_{y\in\sB_{r_{0}}}\Lyap_{k}(y)}\Exp_{x,k}^{v_{n}^{*}}\left[ e^{\int_{0}^{\uuptau_{r_{0}}}\theta\ell(X_s,S_s)}\Lyap(X_s,S_s)\D{s} \right]
\\
& \le \frac{\max_{k\in\cS}\sup_{y\in\sB_{r_{0}}}\psi_{n,k}(y)}{\min_{k\in\cS}\inf_{y\in\sB_{r_{0}}}\Lyap_{k}(y)}\left(\Exp_{x,k}^{v_{n}^{*}}\left[ e^{\int_{0}^{\uuptau_{r_{0}}}\ell(X_s,S_s)}\Lyap(X_s,S_s)\D{s} \right]\right)^{\theta}\\
&\le \kappa_{2} \Lyap_{k}^{\theta}(x),
\end{align*} 
for some positive constant $\kappa_{2}$ (using Harnack's inequality one can ensure that $\kappa_{2}$ is independent of $n$), where in the last inequality we have used \cref{A1.1}. Similar result holds for $\ell$ replaced by $\gamma$ under \cref{A1.2}. Therefore, we obtain $\psi_{n,k} \le \kappa_{2} \Lyap_{k}^{\theta}$ for all sufficiently large $n\in\NN$, $k\in\cS$. This implies that $\psi_{k}^{*} \le \kappa_{2}\Lyap_{k}^{\theta}$ for all $k\in\cS$. Since $\theta \in(0,1)$, this proves (a). This completes the proof.
\end{proof}

Next we obtain stochastic representation formula
for $\psi^*$ which will play crucial role in establishing uniqueness results.
\begin{lemma}\label{L-new}
Suppose that either \cref{A1.1} or \cref{A1.2} holds. Let $\psi^*$ be the eigenfunction obtained by \cref{L3.3}. There exists 
$r_\circ$ such that   for any $v\in\Usm$ we have
\begin{equation}\label{Ln-01}
\psi^*_{k}(x) \, \leq  \, \Exp_{x,k}^{v}\left[ e^{\int_0^{\uuptau_{r}} (c(X_s,S_s,v(X_s,S_s))-\lamstr) \D{s}}\,
\psi^*(X_{\uuptau_{r}}, S_{\uuptau_{r}})\right] \quad \forall\,x\in \sB_{r}^c\,,\,\,k\in\cS\,,
\end{equation}
for all $r\geq r_\circ$, where $\uuptau_{r}$ is the hitting time to $\sB_{r}$.
Furthermore, for any minimizing selector $v^*$ of \cref{EL3.3A} and all $r > r_{0}$, we have
\begin{equation}\label{EL3.3B}
\psi_{k}^*(x)=\Exp_{x,k}^{v^*}\left[ e^{\int_0^{\uuptau_{r}} (c(X_s,S_s,v^*(X_s,S_s))-\lamstr)\,\D{s}}
\psi^*(X_{\uuptau_{r}}, S_{\uuptau_{r}})\right] \quad \forall\,x\in \sB_{r}^c,\, k\in\cS.
\end{equation}
\end{lemma}

\begin{proof}
We prove \eqref{Ln-01} under \cref{A1.1}. An analogous argument also go through under \cref{A1.2}. Recall from the proof of \cref{L3.3} that $\psi^*\in \Sobl^{2,p}(\Rd\times \cS)\cap\order(\Lyap^\theta)$, $p>d$, for some $\theta\in (0, 1)$. In particular, $\psi^*\leq \kappa\Lyap^\theta$ for some $\kappa>0$. Choose $r_\circ$ large enough so that $\max_{\zeta\in \Act} c(x, \zeta)-\lamstr\leq \theta \ell(x)$ for all $|x|\geq r_\circ$ and $k\in\cS$. Fix any $r\geq r_\circ$ and $v\in\Usm$. We may also assume that $\cK\subset \sB_{r_\circ}$ where $\cK$ is given by \cref{A1.1}.
Under \cref{A1.1}, applying It\^{o}-Krylov formula (as in \cref{L3.2}), it follows that
\begin{align*}
\Lyap_{k}(x) \,& \ge\, \Exp_{x,k}^{v}\left[e^{\int_0^{T\wedge\uuptau_r} (\ell(X_s,S_s)-g(X_s,S_s)) \,\D{s}}\,
 \Lyap(X_{T\wedge\uuptau_r},S_{T\wedge\uuptau_r})\right]\\
\,& \ge\, \bigl(\inf_{\Rd} \,\min_{k\in\cS}\,\Lyap_{k}\bigr)\,
\Exp_{x,k}^{v}\left[e^{\int_0^{T\wedge\uuptau_r} \ell(X_s,S_s) \,\D{s}}\,\right] \quad\forall\,x\in\sB_{r}^c\,,\,\, T > 0,\,\, k\in\cS\,. 
\end{align*}
Since $\Prob^v_x(\uuptau_{r} < \infty)=1$, letting $T\to\infty$, by Fatou's lemma we obtain
\begin{equation}\label{EL3.3LyapBd}
\Lyap_{k}(x)\, \ge\, \bigl(\inf_{\Rd} \,\min_{k\in\cS}\,\Lyap_{k}\bigr)\,
\Exp_{x,k}^{v}\left[e^{\int_0^{\uuptau_r} \ell(X_s,S_s) \,\D{s}}\,\right] \quad\forall\,x\in\sB_{r}^c\,,\,\, k\in\cS\,. 
\end{equation}
Since 
$$ \Exp_{x,k}^{v}\left[ e^{\int_0^{\uuptau_{r}\wedge\uptau_n} \ell (X_s,S_s) \D{s}}\,\right]= \Exp_{x,k}^{v}\left[ e^{\int_0^{\uuptau_{r}} \ell (X_s,S_s) \D{s}}\,\Ind_{\{\uuptau_{r}<\uptau_n\}}
\right] + \Exp_{x,k}^{v}\left[ e^{\int_0^{\uptau_n} \ell (X_s,S_s) \D{s}}\,
\Ind_{\{\uptau_n<\uuptau_{r}\}}\right],$$
and the first two expressions converges to the same limit as $n\to\infty$, by monotone convergence theorem, we must have
\begin{equation}\label{EL3.3RA}
\lim_{n\to\infty} \Exp_{x,k}^{v}\left[ e^{\int_0^{\uptau_n} \ell (X_s,S_s) \D{s}}\,\Ind_{\{\uptau_n<\uuptau_{r}\}}\right]=0\,.
\end{equation}
For $m\geq 1$, let us define $\Gamma(m, n)=\{(x, k) \in\sB_n\times\cS \; :\; \psi^*(x, k)\geq m\}$. Since $(\cA^{v} \psi^*)_k(x) \geq \lamstr\psi_k^*(x)$ in $\Rd$ for all $k\in\cS$, applying It\^{o}-Krylov formula  we obtain
\begin{align}\label{EL3.3RB}
\psi^*_k(x) & \leq
\Exp_{x,k}^{v}\left[ e^{\int_0^{\uuptau_{r}\wedge\uptau_n} (c(X_s,S_s,v(X_s,S_s))-\lamstr) \D{s}}\, \psi^*(X_{\uuptau_{r}\wedge\uptau_n}, S_{\uuptau_{r}\wedge\uptau_n})\right]\nonumber
\\
&= \Exp_{x,k}^{v}\left[ e^{\int_0^{\uuptau_{r}} (c(X_s,S_s,v(X_s,S_s))-\lamstr) \D{s}}\, \psi^*(X_{\uuptau_{r}}, S_{\uuptau_{r}}) \Ind_{\{\uuptau_{r}<\uptau_n\}}\right]\nonumber\\
&\qquad + \Exp_{x,k}^{v}\left[ e^{\int_0^{\uptau_n} (c(X_s,S_s,v(X_s,S_s))-\lamstr) \D{s}}\,\psi^*(X_{\uptau_n}, S_{\uptau_n})\Ind_{\{\uptau_n<\uuptau_{r}\}} \right]\,.
\end{align}
Let us now compute the last term of \eqref{EL3.3RB} as follows.
\begin{align}\label{EL3.3RC}
& \Exp_{x,k}^{v}\left[ e^{\int_0^{\uptau_n} (c(X_s,S_s,v(X_s,S_s))-\lamstr) \D{s}}\,\psi^*(X_{\uptau_n}, S_{\uptau_n})\Ind_{\{\uptau_n<\uuptau_{r}\}} \right]\nonumber\\
&\quad \leq m \Exp_{x,k}^{v}\left[ e^{\int_0^{\uptau_n} \theta\ell(X_s, S_s) \D{s}}\,\Ind_{\{\uptau_n<\uuptau_{r}\}} \right]\nonumber\\
&\qquad + \Exp_{x,k}^{v}\left[ e^{\int_0^{\uptau_n}   \theta\ell(X_s,S_s) \D{s}}\,\psi^*(X_{\uptau_n}, S_{\uptau_n})\Ind_{(X_{\uptau_n},S_{\uptau_n})\in\Gamma(m,n)}\Ind_{\{\uptau_n<\uuptau_{r}\}} \right]\nonumber\\
&\quad \leq m \Exp_{x,k}^{v}\left[ e^{\int_0^{\uptau_n} \theta\ell(X_s, S_s) \D{s}}\,\Ind_{\{\uptau_n<\uuptau_{r}\}} \right] \nonumber\\
&\qquad + \kappa_2 \Exp_{x,k}^{v}\left[ e^{\int_0^{\uptau_n}  \theta\ell(X_s,S_s) \D{s}}\,
(\Lyap(X_{\uptau_n}, S_{\uptau_n}))^\theta\Ind_{(X_{\uptau_n},S_{\uptau_n}) \in\Gamma(m,n)}\Ind_{\{\uptau_n<\uuptau_{r}\}} \right]\nonumber\\
&\quad \leq m \Exp_{x,k}^{v}\left[ e^{\int_0^{\uptau_n} \theta\ell(X_s, S_s) \D{s}}\,\Ind_{\{\uptau_n<\uuptau_{r}\}} \right]\nonumber\\
&\qquad + \kappa_2 \left[\frac{m}{\kappa_2}\right]^{\frac{\theta-1}{\theta}} \Exp_{x,k}^{v}\left[ e^{\int_0^{\uptau_n}  \theta\ell(X_s,S_s) \D{s}}\, \Lyap(X_{\uptau_n}, S_{\uptau_n})\Ind_{\{\uptau_n<\uuptau_{r}\}} \right] \nonumber\\
&\quad \leq m \Exp_{x,k}^{v}\left[ e^{\int_0^{\uptau_n} \theta\ell(X_s, S_s) \D{s}}\, \Ind_{\{\uptau_n<\uuptau_{r}\}} \right] 
+ \kappa_2 \left[\frac{m}{\kappa_2}\right]^{\frac{\theta-1}{\theta}}\Lyap_k(x)\to 0,
\end{align}
first letting $n\to\infty$ and then letting $m\to\infty$ and using \eqref{EL3.3RA}, where in the last line we also use
\eqref{EL3.3LyapBd}. Thus letting $n\to\infty$ in \eqref{EL3.3RB}, we deduce that
\begin{equation}\label{EL3.3F}
\psi_{k}^*(x) \le \Exp_{x,k}^{v}\left[ e^{\int_0^{\uuptau_{r}} (c(X_s,S_s, v(X_s,S_s))-\lamstr)\,\D{s}}
\psi^*(X_{\uuptau_{r}}, S_{\uuptau_{r}})\right] \quad \forall\,x\in \sB_{r}^c,\,\,\, k\in\cS.
\end{equation}
This gives us \eqref{Ln-01}.

To prove \eqref{EL3.3B} we choose
any minimizing selector $v^{*}$ of \cref{EL3.3A}. Applying It\^{o}-Krylov formula and Fatou's lemma, we get
\begin{equation}\label{EL3.3D}
\psi_{k}^*(x) \ge \Exp_{x,k}^{v^*}\left[ e^{\int_0^{\uuptau_{r}} (c(X_s,S_s, v^*(X_s,S_s))-\lamstr)\,\D{s}}
\psi^*(X_{\uuptau_{r}}, S_{\uuptau_{r}})\right] \quad \forall\,x\in \sB_{r}^c,\,\,\, k\in\cS.
\end{equation}
Combining \eqref{EL3.3F} for $v=v^*$ and \eqref{EL3.3D} we
obtain \eqref{EL3.3B}. Hence the proof.
\end{proof}
In the next theorem, using of the stochastic representation of the eigenfunctions we deduce the uniqueness of the solution of \cref{EL3.3A} in certain class of functions\,.

\begin{theorem}\label{T3.1}
Suppose that either \cref{A1.1} or \cref{A1.2} holds.
Let $\psi\in \Sobl^{2,p}(\Rd)$, $p>d$, $\psi\gg 0$, be a function satisfying
\begin{equation}\label{ET3.1A}
(\cA \psi)_k(x) \le \lamstr \psi_{k}(x), \quad \text{in\ } \Rd,\quad\forall \,\, k\in\cS\,.
\end{equation}
Then we have $\psi=\kappa\psi^*$ for some $\kappa>0$.
\end{theorem}

\begin{proof}
Let $v$ be a minimizing selector of \eqref{ET3.1A}, that is,
$(\cA^v\psi)_k(x)=\lamstr \psi_k$ for all $k\in \cS$.
Lex $r\geq r_\circ$ where $r_\circ$ is obtained by \cref{L-new}. Applying It\^{o}-Krylov formula and Fatou's lemma, it is easily seen that
\begin{equation}\label{ET3.1B}
\psi_{k}(x) \ge \Exp_{x,k}^{v}\left[ e^{\int_0^{\uuptau_{r}} (c(X_s,S_s, v(X_s,S_s))-\lamstr)\,\D{s}}
\psi(X_{\uuptau_{r}}, S_{\uuptau_{r}})\right] \quad \forall\,x\in \sB_{r}^c,\,\,\, k\in\cS.
\end{equation}
We can choose suitable constant $\kappa>0$ so that $\kappa\psi - \psi^*$ on $\sB_{r}$ and for some $k\in\cS$, \, $\kappa\psi_{k} - \psi_{k}^*$ attains its minimum value $0$ in $\sB_{r}$.
Then by \cref{Ln-01} and \cref{ET3.1B} it follows that $\kappa\psi_{k} -\psi_{k}^*\ge 0$  in $\Rd$ and its minimum value is attained in $\sB_{r}$. Since $v$ is a minimizing selector of \cref{ET3.1A}, from \cref{EL3.3A}, we deduce that
\begin{equation*}
(\Lg^{v} (\kappa\psi - \psi^*))_k(x) - (c_{k}(x,v(x,k)) + m_{k,k}(x,v(x,k)) -\lamstr)^-(\kappa\psi_{k} - \psi_{k}^*)(x) \,\le\, 0\quad \text{in\ } \Rd\,.
\end{equation*}
Therefor, by an application of the strong maximum principle \cite[Theorem 9.6]{GilTru} and using the fact that the system is irreducible, we conclude that $\kappa\psi^*=\psi$. This completes the proof.
\end{proof}

Next we characterize $\lamstr$ as the optimal value of our risk-sensitive ergodic optimal control problem. 

\begin{lemma}\label{L3.5}
Suppose that either \cref{A1.1} or \cref{A1.2} holds.
Then $\lamstr=\sE^*(c)$. 
\end{lemma}

\begin{proof}
From \cref{L3.1} we have $\lamstr\le \sE^{*}(c)$.
In order to show the reverse inequality we perturb $c$ as follows:
\begin{itemize}
\item[•] Under \cref{A1.1}, for $k\in\cS$, $m\in\NN$ we define
\begin{equation*}
 \Tilde{c}_{m,k}(x,\xi) \,\df\, c_{k}(x) + \frac{1}{2}(\ell_{k}(x) - c_k(x,\xi))^+\Ind_{\{\sB_m^c\}}\quad \text{in} \,\, \Rd\times\Act.
\end{equation*}
\item[•] Under \cref{A1.2}, for $k\in\cS$, $m\in\NN$, we consider a sequence of smooth functions $\eta_{m}:\Rd\to [0,1]$ satisfying
$\eta_{m} = 1$ in $\sB_{m}$ and $\eta_{m} = 0$ in $\sB_{m+1}^{c}$, then define
\begin{equation*}
\Tilde{c}_{m,k}(x,\xi) \,\df\, c_{k}(x, \xi)\eta_{m}(x) + (1 - \eta_{m}(x))(\hat{\delta} + \max_{k\in\cS} \norm{c_{k}}_{\infty})\,,
\end{equation*}
where $\hat{\delta} < \gamma \,-\, \max_{k\in\cS} \norm{c_{k}}_{\infty}$.
\end{itemize} It is easy to see that $\Tilde{c}_m$ satisfies conditions of \cref{A1.1}, \cref{A1.2}. Thus, in view of \cref{L3.3}, it follows that there exist an eigenpair $(\psi^*_m, \lambda^*_m)$ satisfying \cref{EL3.3A} with $c$ replaced by $\Tilde{c}_m$\,, for each $m\in\NN$\,.

From our construction of perturbed costs $\Tilde{c}_m$, we have there exists a compact set $\sK$ containing the compact set $\cK$ (as in \cref{A1.1}, \cref{A1.2}) such that $\inf_{\xi\in\Act}\Tilde{c}_m(x,\xi) - \lambda^*_m \geq 0$ for all $x\in\sK^c$\,. For example, under \cref{A1.2} we can take $\sK = \bar{\sB}_{m+1}$ and under \cref{A1.1} since $\Tilde{c}_m$ is unbounded one can construct $\sK$ appropriately\,. Hence, by It\^{o}-Krylov formula it is straightforward to verify that
$\inf_{\Rd}\psi^*_{m,k}(x) \ge \inf_{\sK}\psi^*_{m,k}(x) > 0$, for all $k\in\cS$. 
Now, for any minimizing selector $v_{m}^{*}$ of $\cA \psi^*_{m} = \lambda^*_m \psi^*_{m}$, applying It\^{o}-Krylov formula and Fatou's lemma, we get
\begin{align*}
\psi^*_{m,k}(x)\,& \ge\, \Exp_{x,k}^{v_{m}^{*}}\left[e^{\int_0^T (\Tilde{c}_{m}(X_s,S_s,v_{m}^{*}(X_s,S_s))-\lambda^*_m) \,\D{s}}\,
\psi^*_m(X_T,S_T)\right]\\
&\,\ge\, \biggl(\min_{k}\inf_{\Rd}\,\psi^*_{m,k}\biggr)
\Exp_{x,k}^{v_{m}^{*}}\left[e^{\int_0^T (\Tilde{c}_{m}(X_s,S_s,v_{m}^{*}(X_s,S_s))-\lambda^*_m) \,\D{s}}\right] \,.
\end{align*}
Taking logarithm on both sides, dividing by $T$, and letting $T\to\infty$, we deduce that $\lambda^*_m\ge \sE_{x,k}^{v_{m}^{*}}(\Tilde{c}_m)$.
But we already have $\lambda^*_m\le \sE^{*}(\Tilde{c}_m)$ (see \cref{L3.1}(c)). Therefore, we obtain
\begin{equation*}
\lambda^*_m \,=\, \sE^{*}(\Tilde{c}_m) \,\ge\, 0\quad \forall \, m\in\NN\,.
\end{equation*} Let $\hat{\lambda} \,\df\, \lim_{m\to\infty} \lambda^*_m$. Since $\lambda^*_m\ge \sE_{x,k}^{v_{m}^{*}}(\Tilde{c}_m) \geq \sE_{x,k}^{v_{m}^{*}}(c)$, it is easy to see that $\hat{\lambda} \ge \sE^{*}(c)\ge \lamstr$\,.

Now in order to complete the proof, it remains to show that $\hat{\lambda} = \lamstr$. Arguing as in the proof of \cref{L-new} (see, \cref{Ln-01}) we can find a constant $r_{2} > 0$ such that 
\begin{equation}\label{EL3.5A}
\psi^*_{m,k}(x) \,\le\, \Exp_{x,k}^{v}\left[ e^{\int_0^{\uuptau_{r}} (\Tilde{c}_m(X_s,S_s,v(X_s,S_s)) -\lamstr_m) \D{s}}\,
\psi^*_m(X_{\uuptau_{r}}, S_{\uuptau_{r}})\right] \quad \forall\, x\in \sB_{r}^c,\,\, r\ge r_{2},
\end{equation}
for any $v\in\Usm$.
Using Harnack's inequality together with the Sobolev estimate it can be seen that $\{\psi^*_m\}$ is bounded
in $\Sobl^{2,p}(\Rd\times\cS)$. Thus, we can extract a 
subsequence converging to $\psi\in \Sobl^{2,p}(\Rd\times\cS)$
 satisfying
\begin{equation*}
(\cA \psi)_k(x) \,=\, \hat{\lambda}\psi_{k}(x)\quad\text{in}\,\,\Rd\quad\forall\,\, k\in\cS\,.
\end{equation*}
In view of the estimate as in \cref{EL3.3LyapBd} (and the one with $\ell$ replaced by $\gamma$), by dominated convergence theorem, letting $m\to \infty$ in \cref{EL3.5A}, we deduce that
\begin{equation}\label{EL3.5B}
\psi_{k}(x) \,\le\, \Exp_{x,k}^{v}\left[ e^{\int_0^{\uuptau_{r}} (c(X_s,S_s,v(X_s,S_s)) - \hat{\lambda}) \D{s}}\,
\psi(X_{\uuptau_{r}}, S_{\uuptau_{r}})\right] \quad \forall\, x\in \sB_{r}^c,\,\, r\ge r_{2}.
\end{equation}
We can now repeat the proof of
\cref{T3.1} to conclude that $\kappa\psi^* = \psi$. This intern implies $\hat\lambda=\lamstr$. This completes the proof of the lemma\,.
\end{proof}

We are now ready to complete the proof of \cref{T1.1}.

\begin{proof}[Proof of \cref{T1.1}]
Part (a) follows from \cref{L3.3,L3.5} whereas part (c) follows from \cref{T3.1}. Note that the regularity of the
solution can be improved to $\cC^2$ using standard elliptic
regularity theory.

Now we turn to Part (b). Let $\tilde{v}\in\bUsm$. Then arguing as in \cref{L3.3}, there exist principal eigenpair 
$(\psi_{\tilde{v}}^* , \lamstr(c_{\tilde{v}})\in \Sobl^{2,p}(\Rd\times\cS)\cap\order(\Lyap)\times\RR$, $\psi_{\tilde{v}}^* \gg 0,$ satisfying
\begin{equation}\label{ET1.1AP}
(\cA^{\tilde{v}}\psi_{\tilde{v}}^*)_k(x) \,=\, \lamstr(c_{\tilde{v}})\psi_{\tilde{v},k}^*(x)\quad \text{in\ } \Rd\quad \forall \,\, k\in\cS.
\end{equation} Moreover, following the proof of \cref{L3.5} it is easy to see that $\lamstr(c_{\tilde{v}}) = \inf_{x\in\Rd, k\in\cS}\, \sE_{x,k}(c,\tilde{v})\,.$ Thus we obtain  $\lamstr = \sE^*(c) \le \inf_{x\in\Rd, k\in\cS}\, \sE_{x,k}(c,\tilde{v}) = \lamstr(c_{\tilde{v}})$. Since $\tilde{v}\in\bUsm$ and $\lamstr(c_{\tilde{v}})$ is the principal eigenvalue of \cref{ET1.1AP}, from \cref{EL3.3A} it follows that $\lamstr(c_{\tilde{v}}) \le \lamstr$. Therefore, we obtain $\lamstr = \inf_{x\in\Rd, k\in\cS}\, \sE_{x,k}(c,\tilde{v})$. This proves Part (b). 

Next we prove Part (d). Let $v\in\Usm^{*}$ and $(\psi_{v}^* , \lamstr(c_{v}))\in \Sobl^{2,p}(\Rd\times\cS)\cap\sorder(\Lyap)\times\RR$, $\psi_{v}^* \gg 0,$ be the principal eigenpair satisfying
\begin{equation}\label{ET1.1BP}
(\cA^{v}\psi_{v}^*)_k(x) \,=\, \lamstr(c_{v})\psi_{v,k}^*(x)\quad \text{in\ } \Rd\quad \forall \,\, k\in\cS.
\end{equation} Since $v$ is optimal we have $\lamstr(c_{v}) = \lamstr$. Thus
$$(\cA\psi_{v}^*)_k(x)\leq (\cA^{v}\psi_{v}^*)_k(x)
=\lamstr \psi_{v,k}^*(x).
$$
From \cref{T3.1}, it then follows that $\psi_{v}^* = \kappa \psi^*$. Therefore, from \cref{ET1.1A}, \cref{ET1.1BP} we conclude that $v\in \bUsm$. This completes the proof of the theorem. 
\end{proof}

\section{Risk-sensitive control with Near-monotone cost criterion}\label{N-control}

In this section we consider the ergodic risk-sensitive control problem when the running cost satisfies a near-monotone type structural assumption. In addition to 
\hyperlink{A1}{{(A1)}}--\hyperlink{A4}{{(A4)}}, we impose following assumptions on the coefficients.

\begin{itemize}
\item[\hypertarget{B1}{{(B1)}}] The coefficients $a, b, c, m$ are globally bounded. That is,
$$\sup_{(x,\xi)\in\Rd\times\Act}\Bigl[\max_k\norm{a(x, k)}
+ \max_{k}\norm{b_k(x, \xi)} + |c(x, \xi)| + \max_{i, j}|m_{ij}(x, \xi)| \Bigr]\leq C.$$
Furthermore, 
\begin{equation*}
\sum_{i,j=1}^{d} a^{ij}(x,k)\zeta_{i}\zeta_{j}
\,\ge\,C^{-1} \abs{\zeta}^{2} \qquad\forall\, k\in\cS, x\in\Rd\,,
\end{equation*}
\item[\hypertarget{B2}{{(B2)}}] There exists $\varrho>0$ such that 
$$\min_{\xi\in\Act}m_{i, j}(x, \xi)>\varrho\quad \text{for all}\; x\in\Rd,\; \text{and}\; i\neq j.$$
\item[\hypertarget{B3}{{(B3)}}] The drift term $b$ satisfies 
\begin{equation*}
\max_{\xi\in \Act}\frac{\langle b(x, \xi), x\rangle^+}{|x|} \longrightarrow 0\quad \text{as}\quad |x|\to\infty\,.
\end{equation*}
\end{itemize}

Using the above hypothesis we obtain the following bound on the growth of eigenfunctions.
\begin{lemma}\label{L4.1}
Suppose that \hyperlink{B1}{(B1)}-\hyperlink{B2}{(B2)} hold. Let $\psi\in\Sobl^{2,p}(\Rd\times\cS), p>d$, be positive and  satisfy $(\cA\psi)_k(x)=\lambda\psi_k(x)$ for all 
$k\in\cS$ for some $\lambda\in\RR$. Then there exists
a positive constant $\hat\kappa$ such that
\begin{equation}\label{EL4.1A}
\psi_k(x)\leq \psi_k(0) e^{\hat\kappa |x|}\quad \text{for all}\; x\in\Rd, \; \text{and}\; k\in\cS.
\end{equation}
\end{lemma}

\begin{proof}
Let $v^*$ be a minimizing selector of $(\cA\psi)_k(x)=\lambda\psi_k(x)$, that is ,
$(\cA^{v^*}\psi)_k(x)=\lambda\psi_k(x)$. 
For any point $x_0\in \Rd$, we can translate the coordinate 
by defining $\tilde{f}(x)=f(x+x_0)$ for $f=a_k, b_k, v^*_k, c_k, m_{ij}, \psi_k$, and the new coefficients will also satisfy \hyperlink{B1}{(B1)}-\hyperlink{B2}{(B2)}. So we consider the equation $(\cA^{v^*}\psi)_k(x)=\lambda\psi_k(x)$
in the ball $\sB_2(0)$. Applying the Harnack's inequality
\cite[Theorem~2]{Sirakov}, we find a constant $\kappa_1$ so that
\begin{equation}\label{EL4.1B}
\sup_{\sB_1(0)}\, \max_{k\in\in \cS} \psi_k(x)
\leq \kappa_1\, \inf_{\sB_1(0)}\, \min_{k\in\in \cS} \psi_k(x),
\end{equation}
where the constant $\kappa_1$ depends on the constants in
\hyperlink{B1}{(B1)}-\hyperlink{B2}{(B2)} and $\lambda$.
Therefore, an application of Sobolev estimate \cite[Theorem~9.11]{GilTru} and \eqref{EL4.1B} gives
$$\sup_{\sB_{\frac{1}{2}}(0)}|\grad \psi_k(x)|
\leq \kappa_2 \inf_{\sB_1(0)}\, \min_{k\in\in \cS} \psi_k(x)
\leq \kappa_2 \min_{k\in\in \cS} \psi_k(0),
$$
where the constant $\kappa_2$ depends on $C, \varrho$.
Since $x_0$ is arbitrary, we obtain
$$\sup_{x\in\Rd}\, \max_{k\in\cS} \, \frac{|\grad\psi_k(x)|}{\psi_k(x)}\leq \hat\kappa.$$
This gives us \eqref{EL4.1A}, completing the proof.
\end{proof}
Using \hyperlink{B3}{(B3)}, in the next lemma, we show that $\Exp_{x,k}^{Z}\left[|X_t|\right] \in \sorder{(t)}$ for any $Z\in \Uadm$\,. The proof of the following lemma follows from \cite[Lemma~3.2]{AB18}\,.
\begin{lemma}\label{L4.2}
Suppose that \hyperlink{B1}{(B1)},\hyperlink{B3}{(B3)} hold. Then
\begin{equation}\label{EL4.2A}
\limsup_{t\to\infty}\frac{1}{t}\Exp_{x,k}^{Z}\left[|X_t|\right] \,=\, 0\, \quad \forall\,\,\, Z\in \Uadm\,. 
\end{equation}
\end{lemma}
\begin{proof}
Let $f_{k}(x) = |x|^2$ for all $k\in \cS$ and $x\in \Rd$\,. Thus, by It\^{o}-Krylov formula we deduce that
\begin{align*}
\Exp_{x,k}^{Z}\left[|X_t|^2\right] & \leq |x|^2 + \int_{0}^{t}\left[\langle b(X_s, S_s, Z_s), X_s\rangle^+ + \trace\bigl(a(X_s, S_s)\bigr) + \sum_{j=1}^{N} m_{S_{s^-},j}(X_s,Z_s)|X_s|^2 \right]\D s\\
& = |x|^2 + \int_{0}^{t}\left[ \langle b(X_s, S_s, Z_s), X_s\rangle^+ + \trace\bigl(a(X_s, S_s)\bigr) \right]\D s\,.
\end{align*} Now, closely following the steps as in \cite[Lemma~3.2]{AB18}, we obtain our result\,.
\end{proof}
Recall that a continuous function $f:\Rd\times \Act \to \RR$ is said to be near-monotone with respect to $\lambda \in \RR$, if there exists $\epsilon > 0$ such that the set
$$\{x\in \Rd \mid \min_{\xi\in \Act} f(x,\xi)\leq \lambda + \epsilon \}$$
is either compact or empty\,. In the next theorem, we show that under a near-monotone type structural assumption on the running cost function $c$, ergodic optimal control exists in the space of stationary Markov policies\,.
\begin{theorem}\label{T4.1}
Suppose that \hyperlink{A1}{(A1)}, \hyperlink{A3}{(A3)} and
\hyperlink{B1}{(B1)}-\hyperlink{B3}{(B3)} hold. Also, assume that the running cost function $c$ is near-monotone with respect to $\sE^*(c)$ (see \eqref{Eoptval}). Then there exists $(\psi^*, \lamstr) \in \cC^2(\Rd\times\cS)\times\RR$, $\psi^* \gg 0,$ satisfying
\begin{equation}\label{ET4.1A}
(\cA\psi^*)_k(x) \,=\, \lamstr\psi_{k}^*(x)\quad \text{in\ } \Rd\quad \forall \,\, k\in\cS\,.
\end{equation}
Moreover, the following hold:
\begin{itemize}
\item[(a)] $\lamstr = \sE^*(c)$\,.
\item[(b)] Any $v\in \Usm$ that satisfies
\begin{equation}\label{ET4.1B}
(\cA\psi^*)_k(x) = (\cA^v\psi^*)_k(x)\quad \text{a.e.\ } x\in\Rd\quad \forall \,\, k\in\cS\,,
\end{equation} is stable, and is optimal, that is, $\sE_{x,k}(c, v) = \sE^*(c)$ for all $x\in \Rd$ and $k\in \cS$\,.
\end{itemize}
\end{theorem} 
\begin{proof}
From the proof of \cref{L3.3}, we have there exists $(\psi^*, \lamstr) \in \Sobl^{2,p}(\Rd\times\cS)\times\RR$,\, $p>d$,\, $\psi^* \gg 0,$ satisfying \eqref{ET4.1A}\,. Also, \cref{L3.1}(c) gives us $\lamstr \leq \sE^*(c)$\,. 

Now, let $v\in\Usm$ be a minimizing selector of \eqref{ET4.1A}. Thus, by It\^{o}-Krylov formula and Fatou's lemma, we deduce that for all $T\ge 0$
\begin{align}\label{ET4.1CA}
\psi^*_{k}(x)\, \ge\, \Exp_{x,k}^{v}\left[e^{\int_0^T (c(X_s,S_s,v(X_s,S_s))-\lamstr) \,\D{s}}\, \psi^*(X_T,S_T)\right]\,.
\end{align}
Taking logarithm on both sides of the above inequality, dividing by $T$, and applying Jensen's inequality, it follows that
\begin{align}\label{ET4.1C}
\frac{1}{T}\log\psi^*_{k}(x) + \lamstr \,\ge\, \frac{1}{T}\Exp_{x,k}^{v}\left[\int_0^T c(X_s,S_s,v(X_s,S_s)) \,\D{s}\right]\,+\, \frac{1}{T}\Exp_{x,k}^{v}\left[\log\psi^*(X_T,S_T)\right]\,.
\end{align} From \cref{L4.1}, it is easy to see that $\log\psi_k^*(x) \leq \hat{\kappa}(1 + |x|)$ for all $k\in\cS$\,. Hence, in view of \cref{L4.2}, we obtain $$\limsup_{T\to\infty}\frac{1}{T}\Exp_{x,k}^{v}\left[\log\psi^*(X_T,S_T)\right] = 0\,.$$
Thus, letting $T\to\infty$ in \cref{ET4.1C}, we get
\begin{equation}\label{ET4.1D}
\lamstr \,\ge\, \limsup_{T\to\infty}\frac{1}{T}\Exp_{x,k}^{v}\left[\int_0^T c(X_s,S_s,v(X_s,S_s)) \,\D{s}\right]\,.
\end{equation} Since $\lamstr \leq \sE^*(c)$ and $c$ is near-monotone with respect to $\sE^*(c)$, by the similar argument as in \cite[Lemma~2.1]{AB18}, we conclude that $v$ is stable\,. 

Again, since $c$ is near-monotone with respect to $\sE^*(c)$, one can find a non-empty compact set $\sD$ and a positive constant $\Tilde{\delta}$ such that $\min_{\xi\in\Act}c(x, \xi) - \sE^*(c) > \Tilde{\delta}$ in $\sD^c$\,. Thus, by  It\^{o}-Krylov formula and Fatou's lemma, we obtain
\begin{align}\label{ET4.1E}
\psi^*_{k}(x)\, &\ge\, \Exp_{x,k}^{v}\left[e^{\Tilde{\delta}\uuptau}\, \psi^*(X_{\uuptau},S_{\uuptau})\Ind_{\{\uuptau < \infty\}}\right]\nonumber\\ 
&\ge\, \left(\min_{k\in\cS}\min_{\sD}\psi_k^*\right)\Exp_{x,k}^{v}\left[e^{\Tilde{\delta}\uuptau}\Ind_{\{\uuptau < \infty\}}\right]\quad \forall\,\,\, x\in \sD^c\,,
\end{align}
where $\uuptau \df \uptau(\sD^{c})$\,. By the Harnack's inequality
\cite[Theorem~2]{Sirakov}, we have $\min_{k\in\cS}\min_{\sD}\psi_k^* > 0$. Hence, \cref{ET4.1E} implies that $\min_{k\in\cS}\inf_{\Rd}\psi_k^* > 0$\,. Therefore, from \cref{ET4.1CA}, we obtain 
\begin{align}\label{ET4.1F}
\psi^*_{k}(x)\, \ge\,\left(\min_{k\in\cS}\inf_{\Rd}\psi_k^*\right) \Exp_{x,k}^{v}\left[e^{\int_0^T (c(X_s,S_s,v(X_s,S_s))-  \lamstr) \,\D{s}}\right]\,.
\end{align} Now, taking logarithm on both sides, dividing by $T$ and letting $T\to\infty$, we deduce that
\begin{equation*}
\lamstr \geq \sE_{x,k}(c, v) \geq \sE^*(c)\,.
\end{equation*} Since $\lamstr \leq \sE^*(c)$, it follows that $\lamstr = \sE_{x,k}(c, v) = \sE^*(c)$\,. This completes the proof of the theorem\,.
\end{proof}
\subsection*{Acknowledgement}
This research of Anup Biswas was supported in part by a SwarnaJayanti fellowship DST/SJF/MSA-01/2019-20.

\bibliographystyle{abbrv}
\bibliography{Risk-Switching}

\end{document}